\documentclass[a4paper, article, oneside, UKenglish, 12pt]{article}

\usepackage[utf8]{inputenx} % Source code
\usepackage[T1]{fontenc}    % PDF

\usepackage{lmodern}           % Latin Modern Roman
\usepackage[scaled]{beramono}  % Bera Mono (Bitstream Vera Sans Mono)
\usepackage[final]{microtype}  % Improved typography

\usepackage{amssymb}   % Extra symbols
\usepackage{amsthm}    % Theorem-like environments
\usepackage{thmtools}  % Theorem-like environments
\usepackage{mathtools} % Fonts and environments for mathematical formuale

\usepackage{graphicx}  % Tool for images
\usepackage{babel}     % Automatic translations
\usepackage{csquotes}  % Quotes
\usepackage{textcomp}  % Extra symbols
\usepackage[]{todonotes} %Notes to self

\usepackage{natbib}
\usepackage{varioref}
\usepackage[pdfusetitle]{hyperref}
\usepackage[nameinlink, capitalize, noabbrev]{cleveref}
\crefname{chapter}{Section}{Sections}

\declaretheorem[style = plain]{theorem}
\declaretheorem[style = plain]{corollary}
\declaretheorem[style = plain]{lemma}

\declaretheorem[style = remark,    numbered = no]{remark}

\DeclarePairedDelimiter{\p}{\lparen}{\rparen}   % Parenthesis
\DeclarePairedDelimiter{\sqb}{\lbrack}{\rbrack}
\DeclarePairedDelimiter{\crl}{\lbrace}{\rbrace} % Set
\DeclarePairedDelimiter{\norm}{\lVert}{\rVert}  % Norm

\DeclareMathOperator{\Tr}{Tr}
\DeclareMathOperator{\E}{E}
\DeclareMathOperator{\Var}{Var}

\newcommand{\R}{\mathbb{R}}   % Real numbers
\newcommand{\prconv}{\overset{\Pr}{\to}}
\newcommand{\distconv}{\overset{d}{\to}}
\newcommand{\opr}{o_{p}}
\newcommand{\Opr}{O_{p}}

\newcommand{\thetahat}{\hat{\theta}}
\newcommand{\lambdahat}{\hat{\lambda}}
\newcommand{\alphahat}{\hat{\alpha}}
\newcommand{\betahat}{\hat{\beta}}
\newcommand{\Dhat}{\hat{D}}

\newcommand{\CV}{\textup{CV}}
\newcommand{\TE}{\textup{TE}}
\newcommand{\noti}{_{(-i)}}

\usepackage[left=1in, top=1.05in]{geometry}
\usepackage{setspace}
\begin{document}

\centerline{\large\bf The asymptotic effect of tuning parameters}
\vspace{.4cm} 
\centerline{Ingrid Dæhlen$^{1,2}$, Nils Lid Hjort$^1$ and Ingrid Hobæk Haff$^1$} 
\vspace{.4cm} 
\centerline{$^1$ Department of Mathematics, University of Oslo}
\centerline{$^2$ Norwegian Computing Center, Post box 114 Blindern, Oslo, 0314, Norway}
\vspace{.4cm} 
Address for correspondence: Ingrid Dæhlen, Department of Mathematics, University of Oslo, Moltke Moes vei 35, 0851 Oslo, Norway, email: ingrdae@math.uio.no

\begin{abstract}
	
	Tuning parameters are parameters involved in an estimating procedure for the purpose of reducing the risk of some other estimator. Examples include the degree of penalization in penalized regression and likelihood problems, as well as the balance parameter in hybrid methods. Typically tuning parameters are set to the minimizers of some estimator of the risk, a step which introduces additional randomness and makes standard methodology inapplicable. We derive precise asymptotic theory for this situation. Our framework allows for smooth, but otherwise arbitrary, loss functions and for the risk to be estimated by cross-validation procedures. Results include consistency of the optimal estimator towards a well-defined quantity and asymptotic normality after proper scaling and centring. We give explicit forms and estimators for the limiting variance matrix and results sharply characterizing the distance from the training error to the cross-validated estimator of the risk.
\end{abstract}
\vspace{9pt}
\section{Introduction}\label{section:Introduction}

Suppose we have data $Z_1,\ldots,Z_n$ and wish to estimate a parameter $\theta$. Assume further that for each fixed value of some other parameter $\lambda$ this can be done by maximising a function $\Gamma_n(\theta,\lambda)$ defined in terms of the data. Examples include hybrid estimators where $\Gamma_n(\theta,\lambda) = \lambda\gamma^1_n(\theta) + (1-\lambda)\gamma_n^2(\theta)$ and penalized likelihood methods, where for each $\lambda$, $\theta$ is estimated by the maximizer of $\ell_n(\theta) - \lambda R\p{\theta}$ for some function $R$. This latter situation covers the ridge regression. An example of the former case is the hybrid generative-discriminative estimator defined as the maximizer of $\lambda\ell_n^{\textup{Gen}}(\theta) + (1-\lambda)\ell_n^{\textup{Disc}}(\theta)$ where $\ell_n^{\text{Gen}}$ and $\ell_n^{\textup{Disc}}$ are the log-likelihoods in generative and corresponding discriminative models (see e.g. \citet{bouchard2004tradeoff}).

For each fixed $\lambda$, we can maximize $\Gamma_n(\theta,\lambda)$ with respect to $\theta$ to get the estimator $\thetahat(\lambda)$. Furthermore, if $\Gamma_n$ is sufficiently smooth in $\theta$ and a relatively standard set of regularity conditions is satisfied, $\thetahat(\lambda)$ is consistent for a well-defined quantity $\theta_0(\lambda)$ and $\sqrt{n}\crl{\thetahat(\lambda)-\theta_0(\lambda)}$ converges in law to some known limit distribution. In practice, however $\lambda$ is rarely fixed, and is instead set to minimize some estimated loss, usually approximated by means of an information criterion or cross-validation. There is therefore an additional randomness involved when working with tuning parameters. Because of this, the standard limit results for $\thetahat(\lambda)$ are not immediately applicable to $\thetahat(\lambdahat)$. In practice, this point is often ignored and inference about $\theta$ is made by treating $\lambda$ as fixed at $\lambdahat$ and by using pointwise results for $\thetahat(\lambda)$. Using such a method ignores part of the fitting procedure and added randomness introduced by tuning $\lambda$, however. Because of this, the variance of $\thetahat(\lambdahat)$ tends to be underestimated.

\cref{fig:Motivating example old fits} illustrates the ideas outlined above for a data set concerning female Pima Indians and prevalence of diabetes (publicly available in e.g. the R-package \texttt{mlbench}). In \cref{section:Pima indians} we fit a penalized logistic regression model to these data where the tuning parameter dictating the size of the regularization term is set to the minimizer of the cross-validated Brier score. For each regression coefficient, \cref{fig:Motivating example old fits}  shows histograms based on 2000 non-parametric bootstrap iterations, together with the densities in approximate distributions. The dotted curve corresponds to the distribution we get when using the pointwise results as described in the previous paragraph. The dashed lines correspond to the alternative approximate distributions derived in this article, which take the tuning procedure into account. Note that the curves corresponding to the pointwise estimator seem a bit too narrow, while the curves corresponding to the new estimators are less sharp. This is a consequence of including and not including the effect of the tuning step in the limiting distribution.

\begin{figure}
	\centering
	\includegraphics[width=0.95\textwidth]{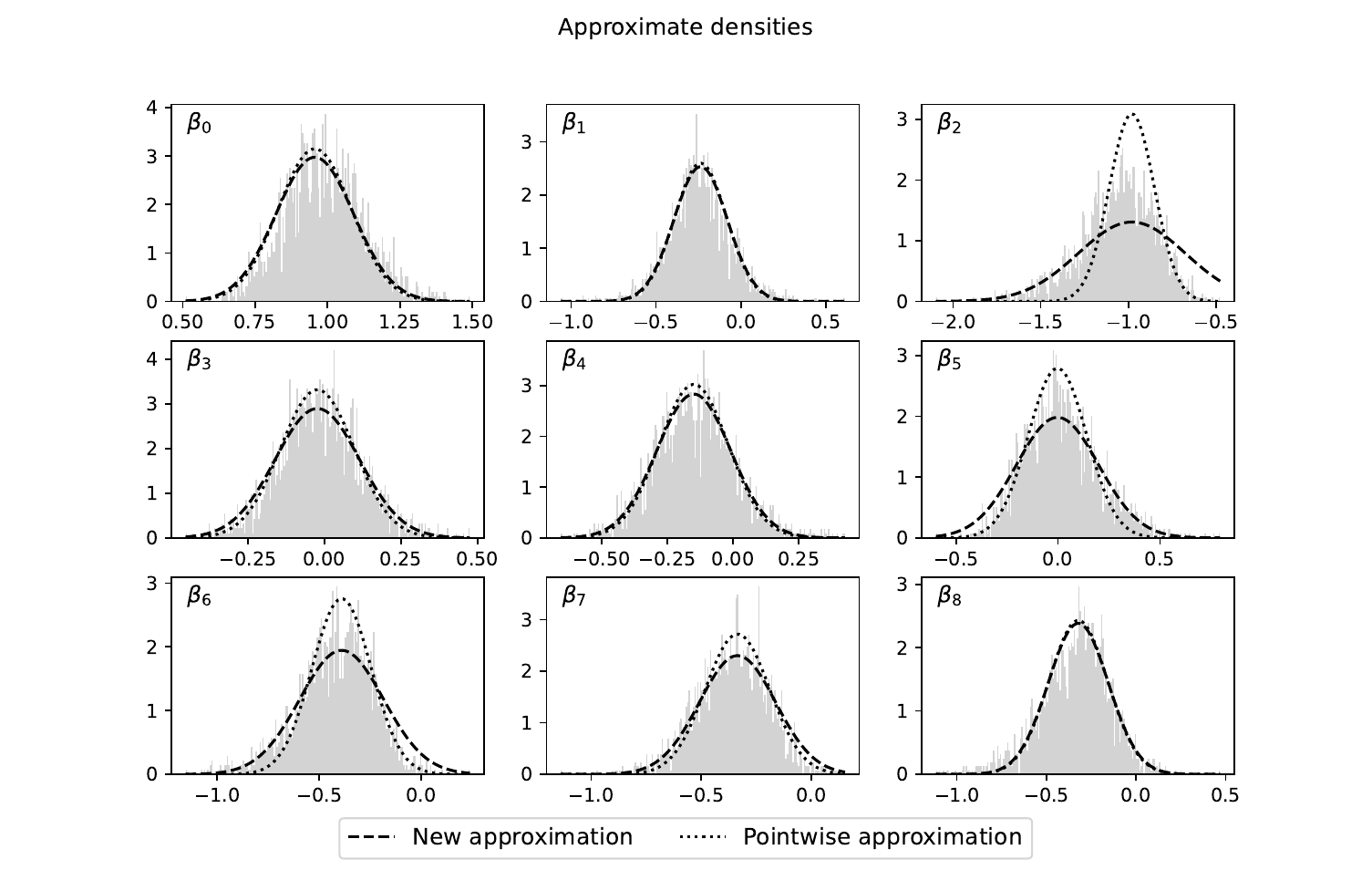}
	\caption{\label{fig:Motivating example old fits} Histograms of draws from the bootstrap distribution of $\betahat(\lambdahat)_j$ for $j=0,\ldots,8$ together with the density in the approximate distribution.}
\end{figure}

In the present article, we fill a gap in the literature by deriving precise limit results for the estimator $\thetahat(\lambdahat)$, allowing inference about $\theta$ to be made while taking the added randomness from tuning $\lambda$ into account. We will start by deriving limit results when $\lambdahat$ is estimated by minimising the relatively simpler training error (\cref{section:Two-staged tuning}) or an information criterion (\cref{section:Information criteria}). We also derive estimators for the limiting variance matrix of $\thetahat(\lambdahat)$ (\cref{section:Estimators}). In \cref{section:Crossvalidation}, we show that the results of \cref{section:Two-staged tuning} hold true also when the cross validated error is used to estimate $\lambda$. This will require us to take a closer look at the cross validation procedure, and in \cref{section:Proving Stone} we show a result proving exactly how far off the training error is from the cross-validated counterpart. In \cref{section: When phi = psi}, we consider a particular situation where the theory simplifies greatly and in \cref{section:Illustrations}, we go though a couple of applications. In particular, we will return to the above analysis of diabetes in Pima Indians in \cref{section:Pima indians}. Lastly, \cref{section:Concluding remarks} contains some concluding remarks.

\subsection{Related work}
% Consult the supplementary material SECTION XXX for an illustration of a situation when ARCONES and related results can be used, and SECTION XXX for a discussion utilizing similar arguments as ARCONES when the limit of $\hat\lambda$ is ill-defined.

To investigate the joint limiting behaviour of $\hat\theta(\hat\lambda)$ and $\hat\lambda$, we reformulate the tuning procedure as a two-stage estimation problem and show that $\thetahat(\lambdahat)$ can be written as components of a Z-estimator (see e.g.~\citet[Ch.~5]{van2000asymptotic}). Many of our proofs are therefore based upon and inspired by results for two-stage estimators, see e.g.~\citet{joe2005asymptotic}, \citet{ko2019model} or \citet{daehlen2024hybrid}, and bears similarities with the framework of generalized estimating equations used frequently in econometrics, see e.g.~\citet{hansen1982large}, \citet{hall2003large} or for a more thorough introduction, \citet[Ch. ~5]{shao2003mathematical}.

%Our results are different from those of \citet{arcones2005convergence} in yet another way, as his results are based on theory of stochastic processes, while we use the framework of Z-estimators. We reformulate the tuning procedure as a two-stage estimation problem and show that $\thetahat(\lambdahat)$ can be written as components of a Z-estimator (see e.g.~\citet[Ch.~5]{van2000asymptotic}). Many of our proofs are therefore based upon and inspired by results for two-stage estimators, see e.g.~\citet{joe2005asymptotic}, \citet{ko2019model} or \citet{daehlen2024hybrid}, and bears similarities with the framework of generalized estimating equations used frequently in econometrics, see e.g.~\citet{hansen1982large}, \citet{hall2003large} or for a more thorough introduction, \citet[Ch. ~5]{shao2003mathematical}.

To apply the two-stage arguments also in the context of cross-validation, we derive sharp results characterizing the difference between the cross-validated estimator of the risk and the training error in \cref{thm:Extension of Stone}. This theorem can be seen as an extension of the result given in \citet{stone1977asymptotic}, where the author writes that the cross-validated likelihood is asymptotically equivalent to the AIC. \citet{stone1977asymptotic} only writes down an informal argument, however, and such intuitive explanations are often repeated when the paper is cited, see e.g.~\citet{claeskens2008model} and \citet{konishi2008information}. To the best of our knowledge, no formal proof of the statement has ever been given, and we therefore believe that our \cref{thm:Extension of Stone} fills this gap in the literature. That being said, recently there has been some work related to approximations to the LOOCV (see \cite{stephenson2020approximate} for an overview), and the theoretical results concerning these approximations bear resemblance to our work (see in particular \cite[Le.~3]{beirami2017optimal}). We still believe our theorem to be novel as it is more general than that of \citet{stone1977asymptotic} since it allows for general $Z$-estimators instead of only maximum likelihood estimators, and the functions used in the cross-validation procedure and estimation of $\theta$ need not be the same.

%To apply the two-stage arguments also in the context of cross-validation, we derive sharp results characterizing the difference between the cross-validated estimator of the risk and the training error in \cref{thm:Extension of Stone}. This theorem can be seen as an extension of the result given in \citet{stone1977asymptotic}, were the author writes that the cross-validated likelihood is asymptotically equivalent to the AIC. \citet{stone1977asymptotic} only writes down an informal argument, however, and such intuitive explanations are often repeated when the paper is cited, see e.g.~\citet{claeskens2008model} and \citet{konishi2008information}. To to the best of our knowledge no formal proof of the statement has ever been given, and we therefore believe that our \cref{thm:Extension of Stone} fills this gap in the literature. Furthermore, our theorem is more general than that of \citet{stone1977asymptotic} as it allows for general $Z$-estimators instead of only maximum likelihood estimators, and the functions used in the cross-validation procedure and estimation of $\theta$ need not be the same.

We also use \cref{thm:Extension of Stone} to discuss whether cross-validation is precise enough to capture the difference between different models, a topic which has been discussed by e.g.~\citet{shao1993linear} or \citet{shao1997asymptotic} in the context of linear models. Furthermore, our theory allows us to characterize the optimism of the training error, something which is studied in e.g.~\citet{efron1983estimating}, \citet{efron1986biased}, \citet{tibshirani1999covariance} or \citet{efron2004estimation}, albeit with a different approach.

Our methodology is in some sense related to the field of post-selection inference, see e.g.~\citet{berk2013valid}, \citet{lee2016exact}, \citet{bachoc2020uniformly} or indeed \citet{kuchibhotla2022post} for a recent review. In this field, the distribution of regression parameters is derived in a way that takes the model selection procedure into account. This is typically done by either conditioning on the model selection procedure or by creating confidence intervals by taking the supremum over all candidate models. We, on the other hand, include the selection of continuous tuning parameters into the fitting procedure itself. Because of this, the field of post-selection inference is related to the present article in spirit only. Furthermore, we provide explicit forms and estimators for the limiting variance of $\thetahat(\lambdahat)$, something which is not obtained though the most popular post-selection inference techniques (see e.g.~the aforementioned review article).

Another topic which bears some similarities with our work is the research concerning asymptotic optimality of different parameter tuning schemes. For a given prediction depending on some tuning parameter $\lambda$, these works investigate its distance to the ``best'' prediction we can possibly make, when $\lambda$ is tuned according to some criterion.  \citet{craven1978smoothing}, \citet{li1985stein} and \citet{speckman1985spline}, \citet{li1986asymptotic} and \citet{li1987asymptotic} all investigate this problem in linear models, and more recently \citet{mu2018asymptotic} and \citet{mu2021asymptotic} have proven asymptotic optimality in a more advanced context. See also \citet{mu2021asymptotic} which contains a more comprehensive list of relevant literature. So far, however, the main concern of this field of research has been whether the tuning procedure ensures that the obtained and optimal predictions are close, rather than the asymptotic effect tuning parameters have on the limit distribution of $\thetahat(\lambdahat)$. We believe this sets our results apart from the literature of asymptotic optimality.

To our knowledge, no results concerning the limiting distribution of $\thetahat(\lambdahat)$ when $\lambdahat$ is set to the minimizer of some general risk estimator exists in the literature. \citet{arcones2005convergence} studies a similar situation, by deriving limit results for $\thetahat(\lambdahat)$ when $\lambdahat$ is set to the minimizer of some estimator of the limiting variance of $\thetahat(\lambda)$. This result is closely related to our theorems, but Arcones' framework only works for one-dimensional parameters $\theta$ and $\lambda$. Furthermore, $\lambdahat$ cannot minimize a general loss function or a cross-validated estimator of the risk. 

Our results are different from those of \citet{arcones2005convergence} in yet another way. \cite{arcones2005convergence}, as well as many others (e.g.~\cite{dodge2000adaptive}, \cite{kato2009asymptotics}, \cite{germain2010weak} and \cite[Se.~10.2]{spokoiny2025sharp}), shows that $S_n(\lambda) = \sqrt{n}\left[\hat\theta(\lambda)-\theta_0(\lambda)\right]$ converges as a stochastic process to a limiting Gaussian $S(\lambda)$. Such results ensure that $\hat\theta(\lambda)$ behaves as expected for all values of $\lambda$ and could in theory be used to prove our results where the asymptotic behaviour of $\hat\lambda$ known. This is pointed out in \cite[Se.~7.2 and 10.2]{dodge2000adaptive}. In practice, however, the behaviour of the tuning parameter is not known, and very few attempts have been made at investigating this. \cite{homrighausen2014leave} and \cite{chetverikov2021cross} have done some work on the Lasso, \cite[sec.~7.2]{dodge2000adaptive} studies the behaviour of $\thetahat(\lambdahat)$ when $\thetahat(\lambda)$ is a trimmed mean in the context of one explicitly defined $\lambdahat$ and \cite{arcones2005convergence} considers a couple of specific tuning schemes, but apart from these papers, virtually no theoretical discussion on how $\hat\lambda$ behaves asymptotically and how this in turn impacts the limiting distribution of $S_n(\hat\lambda)$ seem to be present in the literature. This is especially true when cross-validation is used to tune $\hat\lambda$. We therefore believe the current paper provides a significant contribution to statistical estimation theory. Consult the supplementary material Section S.2 for a discussion utilizing similar arguments as \cite{arcones2005convergence} when no natural limit for $\hat\lambda$ exists.

\subsection{Notation and definitions}\label{section:Notation}
Before we begin with the theory we will introduce notation which will be used throughout this article. For a given function $f(\theta)$ we will let $\nabla_{\theta = \theta_0} f(\theta)$ and $H_{\theta = \theta_0} f(\theta)$ be the gradient and Hessian matrix of $f$ at $\theta_0$ respectively, while $\partial_{\theta = \theta_0}f(\theta) = f'(\theta_0) = \partial f/\partial\theta(\theta_0)$ will denote the Jacobian matrix of $f$ at $\theta_0$. When there is no room for misunderstanding, we will simplify notation and write $\nabla_{\theta_0}f$, $H_{\theta_0}f$ and $\partial_{\theta_0}f$ rather than the above. Furthermore, we will often identify matrices in $\R^{p\times q}$ with vectors in $\R^{pq}$ in the following way $A = (a_1,a_2,\ldots,a_q)\in\R^{p\times q}\leadsto(a_1^T,a_2^T,\ldots,a_q^T)^T\in\R^{pq}$ and hence, if $A$ is a matrix, $\norm{A}= \p{\sum_j\sum_k a_{jk}^2}^{1/2}$. We will also let $\overset{\Pr}{\to}$ and $\overset{d}{\to}$ denote convergence in probability and distribution/law respectively.

We will let $Z_1,\ldots,Z_n\in\R^d$ be independent and identically distributed (i.i.d.) variables from a distribution $F$ and assume that for each fixed value of $\lambda\in\R^q$, $\theta\in\R^p$ is estimated by solving the equation $n^{-1}\sum_{i=1}^{n}\varphi(Z_i,\theta,\lambda) = 0$, for some function $\varphi\colon\R^{d+p+q}\to\R^p$. We will let $\thetahat(\lambda)$ denote the solution. If $\thetahat(\lambda)$ is a maximizer of an expression on the form $n^{-1}\sum_{i=1}^n\phi(Z_i,\theta,\lambda)$ for some function $\phi$, this is satisfied with $\varphi(z,\theta,\lambda) = \nabla_\theta\phi(z,\theta,\lambda)$ under general conditions. In particular, if $\thetahat(\lambda)$ maximises a penalized likelihood of the form $n^{-1}\sum_{i=1}^n\log f(Y_i,\theta) -\lambda P(\theta)$, $\varphi(z,\theta,\lambda)  = u(Y_i,\theta) - \lambda P'(\theta)$ where $u$ is the score function in the model. We will also comment that regression and classification models fit into this framework by considering the covariates as random and letting $Z = (Y,X^T)^T$ where $Y$ is the response and $X$ the predictor. 

Lastly, fix a loss function $\psi\colon\R^{d+p}\to\R$. The training error, which will be denoted by TE$(\lambda)$, is equal to $n^{-1}\sum_{i=1}^n\psi\crl{Z_i,\thetahat(\lambda)}$. The leave-one-out cross validation (LOOCV) estimator of the risk is $n^{-1}\sum_{i=1}^n\psi\crl{Z_i,\thetahat_{(-i)}(\lambda)}$ and will be referenced by CV$(\lambda)$. We will assume that $\lambdahat$ is set to the minimizer of some error estimator and refer to this process as ``tuning'' $\lambda$

\section{Two-stage tuning}\label{section:Two-staged tuning}
We will start by investigating the asymptotic properties of $\thetahat(\lambdahat)$ when $\lambdahat$ is set by tuning with respect to the training error. If this function is convex in $\lambda$ and the minimizer lies in the interior of the parameter space for $\lambda$, this is equivalent to solving TE'$(\lambda)=0$. Hence, under these conditions, $\lambdahat$ satisfies TE'$(\lambdahat) =  \sum_{i=1}^n\thetahat'(\lambdahat)^T\nabla\psi\sqb{Z_i,\thetahat(\lambdahat)}$. If $n^{-1}\sum_{i=1}^n\partial_{\thetahat(\lambdahat)}\varphi(Z_i,\theta,\lambdahat)$ is invertible, the quantity $\thetahat'(\lambdahat)$ exists by the implicit function theorem. Furthermore, it satisfies $n^{-1}\sum_{i=1}^n\sqb{\partial_{\thetahat(\lambdahat)}\varphi(Z_i,\theta,\lambdahat)\thetahat'(\lambdahat) +\partial_{\lambdahat}\varphi\crl{Y_i,\thetahat(\lambdahat),\lambda}}=0$ by the same result. Because of these relations as well as the definition of $\thetahat(\lambda)$, we can express the vector $\alphahat = (\thetahat(\lambdahat)^T,\lambdahat^T,\thetahat'(\lambdahat)^T)^T\in\R^{p+q+pq}$ as the solution to a set of equations. This implies that the full vector $\alphahat$ is what is called a Z estimator, and these are typically consistent for well-defined quantities and asymptotically normal after proper scaling and centring, see e.g.~\citet[Ch.~5]{van2000asymptotic}. Since $\alphahat$ is Z-estimator by the previous arguments, we would expect the same to be true for $\alphahat$ (and in turn also $\thetahat(\lambdahat)$). The following theorem shows that this is indeed correct.

\begin{theorem}\label{thm:alpha TE}
	For each fixed $\lambda$, let $\thetahat(\lambda)$ be the unique solution to $n^{-1}\sum_{i=1}^n\varphi(Z_i,\theta,\lambda)=0$. Furthermore, assume that $\lambdahat$ is the unique solution to TE'$(\lambda)=\opr(1/\sqrt{n})$. Let $\alpha = (\theta^T,\lambda^T,D^T)^T\in\R^{p+q+pq}$ and $\alphahat = (\thetahat(\lambdahat)^T,\lambdahat^T,\thetahat'(\lambdahat)^T)^T$. Define the function $\eta\colon\R^{d+p+q+pq}\to\R^{p+d+pq}$ as $\eta(z,\alpha) = (\eta_1(z,\alpha)^T,\eta_2(z,\alpha)^T,\eta_3(z,\alpha)^T)^T$ where $\eta_1(z,\alpha) = \varphi(z,\theta,\lambda)$, $\eta_2(z,\alpha) =  D^T\nabla_\theta\psi(z,\theta)$ and $\eta_3(z,\alpha) = \partial_\theta\varphi(z,\theta,\lambda)D + \partial_\lambda\varphi(Z,\theta,\lambda)$. Assume the solution to $0 = \Psi(\alpha) = \E\crl{\eta(Z,\theta,\lambda,D)}$ is unique and equal to $\alpha_0$. Then $\alphahat$ is consistent for $\alpha_0$ and furthermore,
	\begin{equation}\label{eq:Influence of alpha TE}
		\alphahat = \alpha_0 - n^{-1}\Psi'(\alpha_0)^{-1}\sum_{i=1}^n\eta\p{Z_i,\theta_0,\lambda_0,D_0}+\opr(1/\sqrt{n}),
	\end{equation}
	provided the function $\Psi$ is continuously differentiable in a neighbourhood of $\alpha_0$, that $\Psi'(\alpha_0)$ is non-singular and that in a neighbourhood of $\alpha_0$ the norm of $\eta(z,\alpha)$ and all first order partial derivatives of $\eta$ with respect to $\alpha$ are bounded by functions $m_0(z)$ and $m_1(z)$ respectively with $\E m_0(Z)^2,\,\E m_1(Z)^2<\infty$. In particular, $\sqrt{n}\p{\alphahat-\alpha_0}$ converges in law to a N$\p{0,V_\alpha}$ distribution where $V_\alpha = \Psi'(\alpha_0)^{-1}$ $\Var\eta(Z,\alpha_0)\crl{\Psi'(\alpha_0)^{-1}}^T$.
\end{theorem}
\begin{remark}
	By Leibniz' integral formula, the function $\Psi$ is smooth in $\alpha$ as long as the same holds true for $\eta$. Furthermore $\alpha_0$ is the unique solution to $\Psi(\alpha) = 0$ as long as $\theta(\lambda)$, defined as the zero of $\E\varphi(Z,\theta,\lambda)$, is well defined for all $\lambda$ and there is only a single $\theta\in\{\theta(\lambda)\,|\,\lambda\in\Lambda)\}$ minimizing the limiting risk, $\E\psi[Z,\theta(\lambda)]$. For a discussion on what happens this condition fails to hold, see Section S2 in the supplementary material.
\end{remark}
\begin{proof}
	By the arguments preceding the theorem, $\sum_{i=1}^n\eta(Z_i,\alphahat) = 0$. Hence, consistency follows from the assumptions and Lemma S1.1 and Corollary S1.2 in the supplementary material. \cref{eq:Influence of alpha TE} follows from \citet[Th.~5.21]{van2000asymptotic}.
\end{proof}

In practice, we are rarely interested in the asymptotic behaviour of the full vector $\alphahat$, but rather in the limiting properties of $\thetahat(\lambdahat)$ alone. 

\begin{corollary}\label{cor:thetahat TE}
	Assume the conditions of \cref{thm:alpha TE} hold true, and for each $\lambda$ let $\theta(\lambda)$ be solution to $\E\varphi(Z,\theta,\lambda) = 0$. Then $\thetahat(\lambdahat)$ is consistent for $\theta_0 = \theta(\lambda_0)$ where $\lambda_0$ is the solution to $\nabla_\lambda\E\psi\crl{Z,\theta_0(\lambda)}=0$. Let $D = \theta'(\lambda_0)$, $Z_1 = H_{\lambda_0} \E\psi\crl{Z,\theta_0(\lambda)}$, $Z_2 = H_{\theta_0}\E\psi(Z,\theta)$, $J = -\partial_{\theta_0}\E\varphi(Z,\theta,\lambda_0)$ and $b = \nabla_{\theta_0}\E\psi(Z,\theta)$, then
	\begin{equation}\label{eq:Influence of thetahat TE}
		\begin{aligned}
		\thetahat(\lambdahat) = \theta_0 &+A_1\frac{1}{n}\sum_{i=1}^n\varphi(Z_i,\theta_0,\lambda_0) + A_2\frac{1}{n}\sum_{i=1}^n\nabla_{\theta_0}\psi\p{Z_i,\theta} \\
		&+ A_3\frac{1}{n}\sum_{i=1}^n\crl{\partial_{\theta_0}\varphi(Z_i,\theta,\lambda_0)D + \partial_{\lambda_0}\varphi(Z_i,\theta_0,\lambda)},
		\end{aligned}
	\end{equation}
	where $A_1 = J^{-1} - DZ_1^{-1}\crl{D^TZ_2 + W}J^{-1}$, $A_2 = -DZ_1^{-1}D^T$ and $A_3=-DZ_1^{-1}M$, where $M$ is a $q\times pq$ block diagonal matrix where the diagonal blocks are each equal to $b^TJ^{-1}$ and $W$ is a $q\times p$ matrix which can be written as $W = (\p{b^TJ^{-1}W^1}^T,\ldots,\p{b^TJ^{-1}W^q}^T)^T$ with
	\begin{equation*}
		W^j = \begin{pmatrix}
			D_j^TH_{\theta_0}\E\varphi^1(Z,\theta,\lambda_0)\\
			D_j^TH_{\theta_0}\E\varphi^2(Z,\theta,\lambda_0)\\
			\vdots\\
			D_j^TH_{\theta_0}\E\varphi^p(Z,\theta,\lambda_0)
		\end{pmatrix} + \partial_{\lambda^j_0}\partial_{\theta_0}\E\varphi(Z,\theta,\lambda),
	\end{equation*}
	where $D_j$ denotes the $j$-th column of $D$. In particular, $\sqrt{n}\crl{\thetahat(\lambdahat)-\theta_0(\lambda_0)}\to\text{N}\{0,A^*$ $K^*\p{A^*}^T\}$ where $A^* = (A_1,A_2,A_3)$ and $K^* = \Var\eta(Z,\theta_0,\lambda_0,D)$ with $\eta$ defined in \cref{thm:alpha TE}.
\end{corollary}
\begin{proof}
	This follows from \cref{thm:alpha TE} and formulas for the inverse of block matrices.
\end{proof}
If $q=1$, we have $M = b^TJ^{-1}$ and $W = b^TJ^{-1}W^1$ making all formulas a lot more readable. Furthermore, the expressions above often simplify greatly, see e.g.~\cref{section: When phi = psi}.

\subsection{Truncated estimators}\label{section:Truncated estimators}
In \cref{cor:thetahat TE} we assume that $\lambdahat$ is found by solving TE'$(\lambda) = 0$. This is, however, not how $\lambda$ is usually tuned in practice. Typically, $\lambdahat$ is set to the minimizer of TE$(\lambda)$ in some compact set $\Lambda$. We will now discuss how \cref{cor:thetahat TE} can be used to make inference also in this situation. For simplicity, we will take $q=1$ and $\Lambda = [0,1]$, but all results should be generalizable to higher dimensions and other parameter sets.

As long as TE is differentiable and $\lambdahat$ lies in the interior of $[0,1]$, minimising TE in $\Lambda$ is equivalent to solving TE'$(\lambda)=0$. Oftentimes, however, TE$(\lambda)$ might be increasing or decreasing with TE'$(\lambda)\neq0$ for all values of $\lambda$. In such cases $\lambdahat$ is set to 0 or 1 depending on which of the boundary points achieves the lowest value of TE$(\lambda)$, and the condition TE'$(\lambdahat)=0$ might not be satisfied. To study this more complicated estimator, we will work with a modified version of $\thetahat(\lambda)$ defined in the following way:
\begin{align}\label{eq:Thetahat_T}
	\thetahat_T = \thetahat(0)I(\lambda\leq0) + \thetahat(\lambda)I\p{0<\lambda<1} + \thetahat(1)I(\lambda\geq1),
\end{align}
where $I$ is the indicator function. This estimator agrees with $\thetahat(\lambda)$ for all $\lambda\in[0,1]$. Outside of the unit interval, however, the estimator is ``truncated'' to either $\thetahat(0)$ or $\thetahat(1)$ depending on which side of $[0,1]$ $\lambda$ lies. Hence, as long as TE$(\lambda)$ is convex and has a global minimizer $\lambdahat_G\in\R$, $\thetahat_T(\lambdahat_G)$ will be equal to $\thetahat(\lambdahat)$ where $\lambdahat$ is the minimizer of TE$(\lambda)$ in $[0,1]$. Furthermore, when the global minimizer of TE$(\lambda)$ lies in $[0,1]$, $\lambdahat$ is the global minimizer ensuring $\thetahat_T(\lambdahat_G) = \thetahat(\lambdahat)$.

\begin{theorem}\label{thm:thetahat truncated TE}
	Assume the conditions of \cref{cor:thetahat TE} hold true, and let $\thetahat_T(\lambda)$ be defined as in \eqref{eq:Thetahat_T}. Furthermore, let $\lambdahat_G$ be the global minimizer of TE$(\lambda)$. The following hold:
	\begin{itemize}
		\item[(a)]
		Assume $\lambdahat_G\prconv\lambda_0<0$ and $\sqrt{n}\crl{\thetahat(0)-\theta_0}\distconv N$, where $N$ is some zero mean normal distribution. Then $\thetahat_T(\lambdahat)\prconv\theta_0$ and $\sqrt{n}\crl{\thetahat_T(\lambdahat)-\theta_0}\distconv N$.
		\item[(b)]
		Assume $\lambdahat_G\prconv\lambda_0>1$ and $\sqrt{n}\crl{\thetahat(1)-\theta_0}\distconv N$. Then  $\thetahat_T(\lambdahat)\prconv\theta_0$ and $\sqrt{n}\crl{\thetahat_T(\lambdahat_G)-\theta_0}\distconv N$.
		\item[(c)]
		Assume $\lambdahat_G\prconv\lambda_0 = 0$ and that $\sqrt{n}\crl{\thetahat(\lambdahat_G)^T-\theta_0^T,\thetahat(0)^T-\theta_0^T,\lambdahat}^T\distconv(N_1^T,N_2^T,N_3)^T$ $\sim\text{N}(0,V)$ for some matrix $V$. Then $\thetahat_T(\lambdahat)\prconv\theta_0$ and $\sqrt{n}\crl{\thetahat_T(\lambdahat)-\theta_0}\distconv I(N_3\geq0)N_1 + I(N_3<0)N_2$.
		\item[(d)]
		Assume $\lambdahat_G\prconv\lambda_0 = 1$ and that $\sqrt{n}\crl{\thetahat(\lambdahat_G)^T-\theta_0^T,\thetahat(1)^T-\theta_0^T,\lambdahat-1}^T\distconv(N_1^T,N_2^T,N_3)^T\sim\text{N}(0,V)$ for some matrix $V$. Then $\thetahat_T(\lambdahat)\prconv\theta_0$ and $\sqrt{n}\crl{\thetahat_T(\lambdahat_G)-\theta_0}\distconv I(N_3\leq0)N_1 + I(N_3>0)N_2$.
	\end{itemize}
\end{theorem}
\begin{remark}
	From Z-estimation theory, $\sqrt{n}\crl{\thetahat(\lambda) - \theta(\lambda)} = J(\lambda)^{-1}n^{-1/2}\sum_{i=1}^n\varphi\crl{Z_i,\theta(\lambda),\lambda}+\opr(1)$ where $\theta(\lambda)$ is the root of $\E\varphi(Z,\theta,\lambda)$ and $J(\lambda) = -\partial_{\theta(\lambda)}\E\varphi(Z,\theta,\lambda)$ for $\lambda=0,1$. Combining this with \eqref{eq:Influence of thetahat TE} and the central limit theorem shows the joint convergence required for case (c) and (d).
\end{remark}
\begin{proof}
	We will only prove (a) and (c). The remaining cases are similar. Since $\lambdahat_G\prconv\lambda_0<0$ the probability of $\lambdahat_G<0$ tends to 1. Furthermore, conditioned on this event, $\thetahat(\lambdahat_G) = \thetahat(0)$. Hence, $\Pr\crl{\sqrt{n}\crl{\thetahat_T(\lambdahat_G) - \theta_0{\lambda_0}}\leq x}=\underset{n\to\infty}{\lim}\sqb{\Pr\crl{\sqrt{n}\crl{\thetahat(0)-\theta_0}\leq x}\Pr\p{\lambdahat_G< 0}  + o\p{\Pr\p{\lambdahat_G\geq0}}}$. The right hand side of this equation converges to $\Pr\p{N\leq x}$. This proves (a).
	
	For (c) note that $\sqrt{n}\crl{\thetahat_T(\lambdahat_G)-\theta_0(0)}=
	I\crl{\sqrt{n}\p{\lambdahat_G-0}\geq 0}\sqrt{n}\crl{\thetahat(\lambdahat_G)-\theta_0(0)} + I\crl{\sqrt{n}\p{\lambdahat_G-0}< 0}\sqrt{n}\crl{\thetahat(0)-\theta_0(0)}$. Now since the function $g(x,y,z) = I(z\geq 0)x + I(z<0)y$ has the following set of discontinuity points $D_g = \crl{(x,y,0)\,|\,x,y\in\R}$, and this set has measure zero in the multivariate normal distribution, the continuous mapping theorem guarantees that $\sqrt{n}\crl{\thetahat_T(\lambdahat_G)-\theta_0(0)}$ converges in law towards $g(N_1,N_2,N_3)=I(N_3\geq0)N_1 + I(N_3<0)N_2$. This concludes the proof.
\end{proof}

\subsection{Estimators of the variance}\label{section:Estimators}
Now that we know how $\thetahat(\lambdahat)$ behaves asymptotically, we will define estimators of its limiting variance. We start by giving approximations to the quantities in \cref{cor:thetahat TE}.

\begin{theorem}\label{thm:Estimators of quantities involved in variance}
	Let $\lambdahat$ be the minimizer of TE$(\lambda)$ and $\thetahat(\lambdahat)$ the solution to $0  = \Phi_n(\theta,\lambdahat)$, where $\Phi_n(\theta,\lambda) = n^{-1}\sum_{i=1}^n\varphi(Z_i,\theta,\lambda)$. Define the following estimators: $\hat{Z}_1 =n^{-1}\sum_{i=1}^n$ $H_{\lambdahat}\psi\crl{Z_i,\thetahat\p{\lambda}}$, $\hat{Z}_2 =H_{\thetahat(\lambdahat)}\Psi_n(\theta)$ and $\hat{b} =\nabla_{\thetahat(\lambdahat)}\Psi_n(\theta)$, where $\Psi_n(\theta) = n^{-1}\sum_{i=1}^n\psi(Z_i,\theta)$. Furthermore, let $\hat{J} = -\partial_{\thetahat(\lambdahat)}\Phi_n(\theta,\lambdahat)$, $\Dhat = \hat{J}^{-1}\partial_{\lambdahat}\Phi_n\crl{\thetahat(\lambdahat),\lambda}$, $\hat{K}^* =n^{-1}\sum_{i=1}^n\hat\eta\hat\eta^T$, where $\hat\eta = \eta\crl{z,\thetahat(\lambdahat),\lambdahat,\thetahat'(\lambdahat)}$ and
	\begin{equation*}
		\hat{W}^j =\frac{1}{n}\sum_{i=1}^n\crl*{\begin{pmatrix}
				\hat D_j^TH_{\thetahat}\varphi^1(Z_i,\theta,\lambdahat)\\
				\vdots\\
				\hat D_j^TH_{\thetahat}\varphi^p(Z_i,\theta,\lambdahat)\\
			\end{pmatrix} + \partial_{\lambdahat^j}\partial_{\thetahat}\varphi(Z_i,\theta,\lambda)}\text{ for }j=1,\ldots,q,\\
	\end{equation*}
	with $\eta$ as defined in \cref{thm:alpha TE}. These estimators are consistent for the corresponding quantities, if there exists functions depending only on $z$ that bound the norms of all partial derivatives of $\partial_{\theta}\varphi$, $\partial_{\lambda}\varphi$, $H_{\lambda}\psi$, $H_{\theta}\psi$, $\nabla_{\theta}\psi$, $(z,\theta,\lambda)\mapsto\eta(z,\theta,\lambda)\eta(z,\theta,\lambda)^T$, $H_{\theta}\varphi^j$ and $\partial_{\lambda^j}\partial_{\theta}\varphi$ for $j=1,\ldots,p$ with respect to $(\theta, \lambda,\theta')$ in a neighbourhood of $(\theta_0^T,\lambda_0^T,D^T)^T$.
\end{theorem}
\begin{proof}
	This follows from standard theory, see e.g.~\citet[Le.~4,  Ap.~B]{daehlenaccurate}.
\end{proof}
We are now ready to define estimators of the limiting variance of $\sqrt{n}\crl{\thetahat(\lambdahat)-\theta_0(\lambda_0)}$.

\begin{theorem}\label{thm:Estimators of variance}
	Assume we are minimising TE$\p{\lambda}$ in [0,1] and the conditions of \cref{thm:Estimators of quantities involved in variance} and \cref{cor:thetahat TE} hold true. Then, $\sqrt{n}\crl{\thetahat(\lambdahat)-\theta_0}$ is asymptotically normal with variance $V$ which can be estimated consistently by $\hat V_1 = \hat A^*\hat K^*\p{\hat A^*}^T$ if $\lambdahat$ lies in the interior of $[0,1]$, or $\hat V_2 = \hat J^{-1}\hat K \p{\hat J^{-1}}^T$ when TE is minimised in $\lambda$ with TE'$(\lambda)\neq0$ for $\lambda$ equal to $0$ or $1$. Here $\hat K = n^{-1}\sum_{i=1}^n\varphi\crl{Z_i,\thetahat(\lambdahat),\lambdahat}\varphi\crl{Z_i,\thetahat(\lambdahat),\lambdahat}^T$, $\hat{A}^* = (\hat{A}_1,\hat{A}_2,\hat{A}_3)$ with $\hat{A}_1 = \hat{J}^{-1} - \hat D\hat{Z}_1^{-1}\crl{\hat D^T\hat{Z}_2 + \hat{W}}\hat{J}^{-1}$, $\hat A_2 = -\hat D \hat Z_1\hat D^T$ and $\hat{A}_3 = -\hat D\hat{Z}_1^{-1}\hat{M}$, where $\hat{M}$ is a block diagonal matrix with $\hat{b}^T\hat{J}^{-1}$ on each of its entries and $\hat{W} = (\p{\hat{b}^T\hat{J}^{-1}\hat{W}^1}^T,\ldots,\p{\hat{b}^T\hat{J}^{-1}\hat{W}^q}^T)^T$. 
\end{theorem}
\begin{proof}
	This follows directly from \cref{cor:thetahat TE} and \cref{thm:thetahat truncated TE}, in conjunction with \cref{thm:Estimators of quantities involved in variance} and the continuous mapping theorem. 
\end{proof}

It could of course happen that $\lambda_0$ is exactly equal to either of the endpoints of $\Lambda$. Then, case (c) or (d) of \cref{thm:thetahat truncated TE} applies and $\sqrt{n}\crl{\thetahat(\lambdahat)-\theta_0}$ is not guaranteed to be asymptotically normal, so that approximating it by N$(0,\hat V_1)$ or N$(0,\hat V_2)$ will not be correct. In \cref{section: When phi = psi} we will, however, see that the distribution in cases (c) and (d) often simplifies to a central multivariate normal distribution also in this case, and that the variance can be estimated by either $\hat V_1$ or $\hat V_2$, defined in the preceding theorem.

\section{Information criteria}\label{section:Information criteria}
The theory presented in the previous sections is not limited to cases where $\lambdahat$ is set to the minimizer of the training error. The results also hold true when the tuning parameter minimises quantities which are ``almost'' equal to TE. The most obvious example of this is perhaps when $\lambdahat$ is set to the minimizer of an information criterion. 

Let $f_\lambda(z,\theta)$ be some parametric model and $\ell_{\lambda,n}\p{\theta}$ the corresponding likelihood. Furthermore, for each $\lambda$, let $\thetahat(\lambda)$ be the maximizer of $\ell_{\lambda,n}\p{\theta}$. Then by definition Akaike's information criterion (AIC), the Bayesian information criterion (BIC) and Takeuchi's information criterion (TIC) take the following forms for each fixed $\lambda$: $\textup{AIC}(\lambda) = -n^{-1}\ell_{\lambda,n}\crl{\thetahat(\lambda)} + n^{-1}p$, $\textup{BIC}(\lambda) = -n^{-1}\ell_{\lambda,n}\crl{\thetahat(\lambda)} + n^{-1}p\log n$ and $\textup{TIC}(\lambda) = -n^{-1}\ell_{\lambda,n}\crl{\thetahat(\lambda)} + n^{-1}\Tr\{\hat{J}(\thetahat,\lambda)^{-1}$ $\hat{K}(\thetahat,\lambda)\}$, where $p$ is the dimension of $\theta$, the number of parameters in the model $f_\lambda(z,\theta)$ at a fixed $\lambda$. Hence, provided $\hat{J}$ and $\hat{K}$ are sufficiently regular, the above information criteria are all only $\opr(1/\sqrt{n})$ away from $n^{-1}\ell_{\lambda,n}\crl{\thetahat(\lambda)}$, allowing \cref{thm:alpha TE} to be applied after replacing $\varphi$ by $\nabla_\theta\log f_\lambda(z,\theta)$ and $\eta_2(z,\alpha)$ by $(z,\alpha) \mapsto\partial_\theta\log f_\lambda(z,\theta)D + \partial_\lambda\log f_\lambda(z,\theta)$. \cref{cor:thetahat TE} and \cref{thm:thetahat truncated TE} also hold after making similar modifications.

The fact that the results of this article can be applied also when $\lambda$ is tuned by minimising the AIC, BIC or TIC follows more or less directly from the forms of the criteria. We will now discuss the focused information criterion (FIC), a newer criterion for which the result is less immediate. The FIC was introduced in \citet{Claeskens2003}, and is an information criterion with a different aim than the classic criteria discussed above. Rather than ranking models according to how well they fit the data overall, the FIC evaluates models by how well they estimate some pre-specified parameter of interest. This parameter of interest is called the focus parameter and should be chosen to reflect the main goal of an analysis. If, for instance, the goal is to estimate the median in a population, the FIC prefers models for which the estimator of the median is good. The quality of an estimator is evaluated by its mean squared error (MSE), and hence, the FIC of a model is the estimated MSE of that model's estimator of the focus parameter. We will work with the relatively newer version of the criterion introduced in \citet{jullum2017parametric} and modified in \citet{daehlenaccurate}.

As before, let $f_\lambda(z,\theta)$ be some parametric model and $\thetahat(\lambda)$ the maximizer of $\theta\mapsto\ell_{\lambda,n}\p{\theta} = \sum_{i=1}^n\log f_\lambda(Z_i,\theta)$ for each $\lambda$. Let $\mu$ be the focus parameter with true value $\mu_0$ and  assume $\mu$ takes the value $\mu_\lambda(\theta)$ in the model $f_\lambda\p{z,\theta}$. The FIC of the model $f_\lambda\p{z,\theta}$ for each $\lambda$ is the MSE of $\mu_\lambda\crl{\thetahat(\lambda)}$. \citet{daehlenaccurate} gives an estimator of this quantity: $\textup{FIC}(\lambda) = \hat{b}_\lambda^2 + n^{-1}\p{2\hat{b}_\lambda\hat{c}_\lambda - \hat{\kappa}_\lambda + \hat{\tau}_\lambda}$, where $\hat{c}_\lambda$, $\hat\kappa_\lambda$ and $\hat\tau_\lambda$ are consistent estimators of certain population quantities defined in the paper, and $\hat{b}_\lambda = \mu\crl{\thetahat(\lambda)} - \hat{\mu}_0$ for some estimator $\hat\mu_0$, which is assumed to be consistent for $\mu_0$. Assume now that $\hat\mu_0$ is the solution to an equation on the form $0 = n^{-1}\sum_{i=1}^n\xi(Z_i,\mu)$. We then have $\textup{FIC}(\lambda) = n^{-1}\sum_{i=1}^n\sqb{\mu\crl{\thetahat(\lambda)} - \hat\mu_0}^2 + \Opr(1/n)$, and hence, provided sufficient regularity, \cref{thm:alpha TE} can be applied by replacing $\theta(\lambda)$ by $\p{\theta(\lambda)^T,\mu}^T$, $\varphi(z,\theta,\lambda)$ by $\p{\varphi(z,\theta,\lambda)^T,\xi(z,\mu_0)}^T$ and $\eta_2(z,\alpha)$ by $(\theta,\mu,\lambda,D)\mapsto2\crl{\mu\p{\theta}-\mu}\p{D^T\nabla_\theta\mu(\theta)^T,\,-1}^T$. In \cref{section:Hybrid likelihood}, we illustrate this in an example.

\begin{remark}
	It is worth noting that we require the criterion used for fitting $\lambda$ to be differentiable as a function of the tuning parameter. Because of this, our theorems do not cover situations where $\lambda$ is discrete. In particular, our theorems cannot be applied when an information criterion is used to e.g.~choose between a finite number of models.
\end{remark}

\section{Cross-validation}\label{section:Crossvalidation}
When $\lambda$ is tuned by minimising the training error, $\thetahat(\lambdahat)$ can be written as the components of a Z-estimator. Because of this the results of the previous section follow more or less directly from standard theory. In practice, however, $\lambda$ is seldom tuned with respect to TE, but instead set to the minimizer of CV$(\lambda)$, the estimate of the risk function obtained by cross-validation. In this section, we will show that the difference between tuning with respect to TE and CV is negligible asymptomatically and that all of the previously derived results can be applied also when $\lambdahat$ is the minimizer of CV. 

First, note that for each $\lambda$ the vector $(\thetahat(\lambda)^T,\thetahat'(\lambda)^T)^T$ is the solution to equation $n^{-1}\sum_{i=1}^n\xi(Z_i,\theta,\lambda,D) = 0$ with $\xi(z,\theta,\lambda,D) = (\xi_1(z,\theta,\lambda)^T,\xi_2(z,\theta,\lambda,D)^T)^T$ where $\xi_1 = \varphi$ and $\xi_2(z,\theta,\lambda,D)  = \partial_{\theta}\varphi(z,\theta,\lambda)D + \partial_{\lambda}\varphi(z,\theta,\lambda)$.  This follows from definition of $\thetahat(\lambda)$ and the implicit function theorem. Because of this, the vector $(\thetahat(\lambda)^T,\thetahat'(\lambda)^T)^T$ is a Z-estimator, and under weak conditions $(\thetahat(\lambda)^T,\thetahat'(\lambda)^T)^T$ is consistent for some vector $\p{\theta_0(\lambda)^T,D_0(\lambda)^T}^T$ and $(\thetahat(\lambda)^T, \thetahat'(\lambda)^T)^T\approx(\theta_0(\lambda)^T, D_0(\lambda)^T)^T + V(\lambda)^{-1}n^{-1}\sum_{i=1}^n\xi\crl{Z_i,\theta_0(\lambda), D_0(\lambda),\lambda}$ where $V(\lambda) = -\partial_{\theta_0(\lambda),D_0(\lambda)}\E\xi(Z,\theta,\theta',\lambda)$. See e.g.~chapter 5 of \citet{van2000asymptotic} for proofs and sufficient conditions. Because of this $(\thetahat(\lambda)^T-\thetahat_{(-i)}(\lambda)^T, \thetahat'(\lambda)^T - \thetahat'_{(-i)}(\lambda)^T) \approx n^{-1}V(\lambda)^{-1}\xi\crl{Z_i,\theta_0(\lambda),\theta_0'(\lambda)} = \Opr(1/n)$. We formalize this idea in the following lemma.

\begin{lemma}\label{lemma:Distance thetahati and thetahat}
	Let the parameter sets $\Theta$ for $\theta$, $\mathcal{D}$ for $\thetahat'(\lambda)$ and $\Lambda$ for $\lambda$ be compact and $\xi(z,\theta,D,\lambda) = (\varphi(z,\theta,\lambda)^T,\xi_2(z,\theta,\lambda,D)^T)$ where $\xi_2(z,\theta,\lambda,D) = \partial_{\theta}\varphi(z,\theta,\lambda)D + \partial_{\lambda}\varphi(z,\theta,\lambda)$. Assume the conditions of Lemma S1.3 in the supplementary material hold true for this $\xi$, and in particular, that $m_0(z)$ bounds $\xi(z,\theta,D,\lambda)$ for all $\theta$, $D$ and $\lambda$. Then $\norm{\thetahat_{(-i)}(\lambda) - \thetahat(\lambda)}, \norm{\thetahat'_{(-i)}(\lambda) - \thetahat(\lambda)}\leq a_{0,n}m_0(Z_i)$ where $a_{0,n} = \Opr(1/n)$ uniformly in $\lambda$ and does not depend on $i$.
\end{lemma}

The proof is given in the supplementary material. With \cref{lemma:Distance thetahati and thetahat}, we are ready to quantify the difference between CV'$(\lambda)$ and TE'$(\lambda)$.

\begin{lemma}\label{lemma:Crossvalidation uniformly close training error}
	Assume all conditions of \cref{lemma:Distance thetahati and thetahat} hold true with $\E m_0(Z)^4<\infty$. Then
	\begin{equation}\label{eq:CV and training error distance}
		\frac{1}{n}\sum_{i=1}^n\thetahat_{(-i)}'(\lambda)^T\nabla_{\thetahat_{(-i)}(\lambda)}\psi\p{Z_i,\theta} = \frac{1}{n}\sum_{i=1}^n\thetahat'(\lambda)^T\nabla_{\thetahat(\lambda)}\psi\p{Z_i,\theta} + \delta_n(\lambda),
	\end{equation}
	where $\norm{\delta_n(\lambda)} = \Opr(1/n)$ uniformly in $\lambda$, provided there exist square integrable functions $p_1$ and $p_2$ not depending on $\theta$ that bound respectively all first and second order partial derivatives of $\psi$ with respect to $\theta$.
\end{lemma}
\begin{proof}
	By \cref{lemma:Distance thetahati and thetahat}, $\thetahat_{(-i)}(\lambda) = \thetahat(\lambda) +  R_n^i(\lambda)$ and $\thetahat_{(-i)}'(\lambda) = \thetahat(\lambda) + S_n^i(\lambda)$ with $\norm{R_n^i(\lambda)},\, \norm{S_n^i(\lambda)} \leq a_{0,n}m_0(Z_i)$. A Taylor expansion reveals that the left hand side of \eqref{eq:CV and training error distance} is equal to $n^{-1}\sum_{i=1}^n\crl{\thetahat'(\lambda) + S^i_n(\lambda)}^T\crl{\nabla_{\thetahat(\lambda)}\psi\p{Z_i,\theta} + H_{\theta_i^*}\psi\p{Z_i,\theta}R_n^i(\lambda)}$, for some $\theta_i^*$s on the line segments between $\thetahat(\lambda)$ and $\thetahat_{(-i)}(\lambda)$.  Utilizing the existence of $p_1$ and $p_2$, we get that $\norm{CV'(\lambda)-TE'(\lambda)}$ is bounded by $a_{0,n}\thetahat'(\lambda)^Tn^{-1}\sum_{i=1}^np_2(Z_i)m_0(Z_i) + a_{n,0}n^{-1}\sum_{i=1}^np_1(Z_i)m_0(Z_i) + a_{0,n}^2n^{-1}\sum_{i=1}^np_2(Z_i)m_0(Z_i)^2$. Since $\E p_1(Z)^2$, $\E p_2(Z)^2$ and $\E m_0(Z)^4$ exist, the means converge. This ensures $\norm{CV'(\lambda)-TE'(\lambda)} = \Opr(1/n)$ uniformly in $\lambda$ as $\mathcal{D}$ is compact.
\end{proof}

Since CV'$(\lambda)$ and TE'$(\lambda)$ are uniformly close as functions of $\lambda$ by the above lemma, we would expect $\thetahat(\lambdahat)$ to asymptotically behave similarly when tuning with respect to CV and TE. The following theorem confirms the intuition.

\begin{theorem}\label{thm:thetahat CV asymptotically equivalent thetahat TE}
	Let $\lambda_{\textup{TE}}$ be the minimizer of TE and $\lambda_{\textup{CV}}$ the minimizer of CV and assume that the conditions of \cref{thm:thetahat truncated TE} and \cref{lemma:Crossvalidation uniformly close training error} hold true. Then $\sqrt{n}\crl{\thetahat(\lambda_{\textup{TE}}) - \theta_0}$ and $\sqrt{n}\crl{\thetahat(\lambda_{\textup{CV}}) - \theta_0}$ are asymptotically equivalent. In particular, \cref{thm:alpha TE}, \cref{cor:thetahat TE} and \cref{thm:thetahat truncated TE} hold true for $\thetahat(\lambdahat_{\textup{CV}})$, and the limiting distribution of $\sqrt{n}\crl{\thetahat(\lambda_{\textup{CV}}) - \theta_0}$ can be estimated as in \cref{thm:Estimators of variance}.
\end{theorem}
\begin{proof}
	Let $\alphahat_{\textup{CV}} = (\thetahat(\lambdahat_{\textup{CV}})^T, \lambdahat_{\textup{CV}}^T, \thetahat'(\lambdahat_{\textup{CV}})^T)^T$. Since TE'$(\lambda) =$ CV'$(\lambda)+\Opr(1/n)$ uniformly in $\lambda$, $\textup{TE}'(\lambdahat_{\textup{CV}}) = \Opr(1/n)$.  Hence, $n^{-1}\sum_{i=1}^n\eta(Z_i,\alphahat_{\textup{CV}}) = \Opr(1/n)$. Since $n^{-1}\sum_{i=1}^n\eta(Z_i,\alphahat_{\textup{CV}}) = \opr(1/\sqrt{n})$ is all that is needed for \citet[Th. 5.21]{van2000asymptotic}, the conclusion of \cref{thm:alpha TE} follows for $\alphahat_{\textup{CV}}$ as well. This ensures that $\sqrt{n}\p{\alphahat -\alpha_0}$ and $\sqrt{n}\p{\alphahat_{\textup{CV}} - \alpha_0}$ are asymptotically equivalent. \cref{thm:thetahat CV asymptotically equivalent thetahat TE} follows from this fact.
\end{proof}

\subsection{A more precise approximation}\label{section:Proving Stone}
\cref{thm:thetahat CV asymptotically equivalent thetahat TE} is a somewhat surprising result, as it tells us that asymptotically speaking it makes no difference if we tune with respect to TE or CV. Yet, in most applications, minimising TE and CV gives very different results. We will now take a closer look at what the difference between CV and TE really is. To make the derivations more readable, we will only show pointwise results for $\thetahat(\lambda)$, and drop $\lambda$ from the notation.

Arguing similarly as in the paragraph preceding \cref{lemma:Distance thetahati and thetahat} we can show $\thetahat_{(-i)} -\thetahat \approx -n^{-1}J^{-1}\varphi(Z_i,\theta_0)$ where $\theta_0$ is the solution to $\E\varphi(Z,\theta) = 0$ and  $J = -\partial_{\theta_0}\E\varphi(Z,\theta)$. This approximation is not particularly surprising. The function $z\mapsto J^{-1}\varphi\crl{z,\theta_0}$ is called the influence function of $\thetahat$ as it in some sense tells how much a single data point affects estimation of $\theta$. See e.g.~\citet{huber2004robust} for more details on influence functions. Since $\thetahat_{(-i)}$ is the estimator you get instead of $\thetahat$ if you remove the $i$-th data point, it is natural that the difference between $\thetahat_{(-i)}$ and $\thetahat$, is precisely the ``influence'' of $Z_i$. 

Assuming $\psi$ is sufficiently smooth in $\theta$, a Taylor expansion shows $ n^{-1}\sum_{i=1}^n\psi\p{Z_i,\thetahat} + n^{-1}\sum_{i=1}^n(\thetahat\noti - \thetahat)^T\nabla_{\thetahat}\psi\p{Z_i,\theta}\approx \CV$. A further Taylor expansion of $\nabla_{\thetahat}\psi\p{Z_i,\theta}$ around $\theta_0$, reveals $\CV \approx n^{-1}\sum_{i=1}^n\psi\p{Z_i,\thetahat} - n^{-1}\sum_{i=1}^n(\thetahat\noti-\thetahat)\nabla\psi\p{Z_i,\theta_0}\approx n^{-1}\sum_{i=1}^n\psi\p{Z_i,\thetahat} - n^{-2}\sum_{i=1}^n\varphi(Z_i,\theta_0)^T\p{J^{-1}}^T\nabla_{\theta_0}\psi\p{Z_i,\theta}$. By the law of large numbers and properties of the trace operator, we hence get $\CV\approx \TE - n^{-1}\Tr\p{J^{-1}C}$ where $C = \E\varphi(Z_i,\theta_0)\nabla_{\theta_0}\psi(Z,\theta_0)^T$. The following theorem confirms these informal arguments. A full proof is given in the supplementary material.

\begin{theorem}\label{thm:Extension of Stone}
	Let $\thetahat$ and $\thetahat\noti$ the solutions to $\sum_{j=1}^n\varphi(Z_j,\theta) = 0$ and $\sum_{j\neq i}^n\varphi(Z_j,\theta)$ respectively. Assume further that $\thetahat,\,\thetahat\noti\in\Theta\subseteq\R^p$ where $\Theta$ is compact and that there exists functions $m_0,\, m_1,\,p_1$ and $p_2$ bounding the norms of $\varphi$, all partial derivatives of $\varphi$ with respect to $\theta$, $\psi$ and all partial derivatives of $\psi$ with respect to $\theta$, respectively. Then, for a given function $\psi(z,\theta)$, we have
	\begin{equation}\label{eq:Formel Stone thm}
		n^{-1}\sum_{i=1}^n\psi\p{Z_i,\thetahat\noti} = n^{-1}\sum_{i=1}^n\psi\p{Z_i,\thetahat} - n^{-1}\Tr\p{J^{-1}C} + \opr(1/n)
	\end{equation}
	where $J = -\partial_{\theta_0}\E\varphi(Z,\theta)$ and $C = \E\varphi(Z,\theta_0)\nabla_{\theta_0}\psi(Z,\theta)^T$, with $\theta_0$ being the solution to $\E\varphi(Z,\theta)=0$, provided the eigenvalues of $\crl{\partial_{\theta}\E\varphi(Z,\theta)}^{-1}$ are bounded in $\Theta$ and $\E m_0(Z)^4,\,\E m_1(Z)^4,\,\E p_1(Z)^4,\,\E p_2(Z)^2< \infty$.
\end{theorem}

When $\varphi(z,\theta) = \nabla_{\theta}\log f(z,\theta)$ for some density or probability mass function $f$, $J$ is the Fisher matrix in  the model. Furthermore, $C=K$ if $\psi = \log f$, where $K = \Var\varphi(Z,\theta_0)$. Hence, \cref{thm:Extension of Stone} simplifies to $\sum_{i=1}^n\psi\p{Z_i,\thetahat\noti} = \textup{TIC}+\opr(1)$, where TIC is Takeuchi's information criterion (TIC), defined in \cref{section:Information criteria}. Additionally, if the model $f$ is specified correctly, $K = J$ and $\sum_{i=1}^n\psi\p{Z_i,\thetahat\noti} = \ell_n\p{\thetahat} - p +\opr(1)$. The right hand side of this equation is easily recognized as Akaike's information criterion (AIC), leading \cref{thm:Extension of Stone} to simplify to the results of \citet{stone1977asymptotic} when $\varphi(z,\theta) = \nabla_\theta \log f(z,\theta)$ and the parametric model is specified correctly. 

\cref{thm:Extension of Stone} shows that CV and TE are more similar than one might initially think. In fact, the only thing separating the two quantities is a term of order $\Opr(1/n)$, which is dominated by the size of TE itself, asymptotically speaking. In smaller samples, however, $n^{-1}\Tr\p{J^{-1}C}$ might be quite large and will ensure that tuning with respect to TE and CV give different results. That being said, there are cases where $n^{-1}\Tr\p{J^{-1}C}$ is not negligible and the difference between TE and CV will matter also in the limit. We will now turn our attention towards such a case.

\subsection{When more precision is needed}\label{section:When more precision is needed}
For each parameter $\theta$ let $R(\theta)$ be some risk function, and let $\thetahat$ be some estimator of $\theta$ based on data $Z_1,\ldots,Z_n$. Assume further that $\thetahat\prconv\theta_0$ and $\sqrt{n}\p{\thetahat-\theta_0}\distconv\text{N}\p{0,V}$ for some matrix $V$. By Taylor expanding $\psi$ around $\theta_0$ and using the law of total expectation, one can show that, provided sufficient regularity,
\begin{equation}\label{eq:Risk function Taylor expansion}
	\E\crl{R(\thetahat)} = R(\theta_0) + n^{-1}\sqb{\nabla_{\theta_0}R(\theta_0)c + (1/2)\Tr\crl{H_{\theta_0}R(\theta)V}} + o(1/n),
\end{equation}
where $c$ is equal to the limit of $n\E\p{\thetahat-\theta_0}$, see \citet{daehlenaccurate} for more on this parameter. If $\thetahat\in\R$ and $R(\theta) = (\theta_0 - \theta)^2$, \eqref{eq:Risk function Taylor expansion} reduces to the classical bias-variance decomposition of the MSE. With $R(\theta)$ equal to the error rate of certain classification procedures, \eqref{eq:Risk function Taylor expansion} reduces to the expressions given in \citet{oneill1980general}.

Note that \eqref{eq:Risk function Taylor expansion} consists of three parts: a constant part, $R(\theta_0)$,  a term of size $O(1/n)$ and a remainder term of smaller order. For two estimators $\thetahat_1$ and $\thetahat_2$ consistent for different values $\theta_0^1$ and $\theta_0^2$ respectively, the difference between their expected risk will asymptotically be dominated by how much $R(\theta_0^1)$ and $R(\theta_0^2)$ differ, making the $O(1/n)$ terms in \eqref{eq:Risk function Taylor expansion} negligible asymptotically. If, on the other hand $\theta_0^1$ and $\theta_0^2$ are equal, the difference in expected risk for $\thetahat_1$ and $\thetahat_2$ will be dominated by the difference between the two corresponding terms of size $O(1/n)$ in \eqref{eq:Risk function Taylor expansion}. In this latter situation, we therefore need to estimate the $O(1/n)$-effect from data in order to properly distinguish the risk of $\thetahat_1$ from that of $\thetahat_2$. Both TE and CV attempt to estimate $\E R(\thetahat)$, but, as we will see, neither estimator is precise enough capture the $O(1/n)$ effects.

Let $R_n(\theta) = n^{-1}\sum_{i=1}^n\psi(Z_i,\theta)$ and $\thetahat$ be some estimator consistent for $\theta_0$. Assuming $\psi$ is sufficiently smooth, a Taylor expansion shows
\begin{equation}
	\TE = R_n(\theta_0) + \p{\thetahat-\theta_0}^T\nabla_{\theta_0}R_n(\theta) + (1/2)\p{\thetahat-\theta_0}^TH_{\theta_0}R_n(\theta)\p{\thetahat-\theta_0} + \opr(1/n).
\end{equation}
Since $R_n(\theta_0)$ is a mean and $\thetahat\prconv\theta_0$, we have $\TE\prconv R(\theta_0)$, which means that TE does a good job, as far as estimating the constant part of \eqref{eq:Risk function Taylor expansion} goes. Approximating the term of order $O(1/n)$ is, however, another story. The most obvious reason is perhaps that the remaining terms are nowhere near being consistent for $n^{-1}\sqb{\nabla_{\theta_0}R(\theta_0)c + (1/2)\Tr\crl{H_{\theta_0}R(\theta)V}}$, but even if they were, TE would not be precise enough to properly catch effects of size $O(1/n)$, as $R_n(\theta_0) = R(\theta_0) + \Opr(1/\sqrt{n})$ by the central limit theorem. Since a term of order $\Opr(1/\sqrt{n})$ will dominate $O(1/n)$ effects, the $O(1/n)$-terms in \eqref{eq:Risk function Taylor expansion} are too small to be captured by TE and will be washed out by the error of $R_n(\theta_0)$. As a result, TE is not well suited for evaluating estimators whose limiting risk differs only in the $O(1/n)$-term. One might hope that this is a defect of the TE only, and that CV will be precise enough to capture the $O(1/n)$ effects. Sadly, \cref{thm:Extension of Stone} stops this idea dead in its tracks. Since the difference between CV and TE is a term of order $\Opr(1/n)$, the added precision of CV is not sufficient to correct for the $\Opr(1/\sqrt{n})$ error TE makes. This is pointed out in the context of model selection of linear regression models in \citet{shao1993linear}.

The above might lead one to believe that the difference between CV and TE is negligible and that there is no reason to use CV rather than TE in analysis. This conclusion is, however, slightly too negative, as there is an important difference between CV and TE in expected value. To see this, assume that $\E\psi(Z,\theta_0) = 0$ and that $\thetahat = \theta_0 + J^{-1}n^{-1}\sum_{i=1}^n\varphi(Z_i,\theta_0) + \opr(1/\sqrt{n})$. Then, by the central limit theorem, $\sqrt{n}\p{\thetahat-\theta_0}$ and $\sqrt{n}\crl{R_n(\theta_0)-R(\theta_0)}$ converge jointly in law to a central normal distribution with variance matrix
\begin{equation*}
	\Sigma = \begin{pmatrix}
		J^{-1}K(J^{-1})^T&J^{-1}C^T\\
		C^T\p{J^{-1}}^T&\Var\psi(Z,\theta_0)
	\end{pmatrix},
\end{equation*}
where $J$ and $C$ are as in \cref{thm:Extension of Stone} and $K=\Var\varphi(Z,\theta_0)$. Hence, provided all necessary quantities are uniformly integrable (see e.g.~\citet[p.~31]{billingsley2013convergence}), $\E(\TE) = R(\theta_0) + \Tr\p{CJ^{-1}} + (1/2)\Tr\crl{H_{\theta_0}R(\theta)J^{-1}K(J^{-1})^T} + o(1/n)$ which misses $\E R(\thetahat)$ by a term $n^{-1}\Tr\p{J^{-1}C} + o(1/n)$. Because of this, TE is a biased estimator of $\E R(\thetahat)$. CV, on the other hand, satisfies $\CV = \TE -\Tr\p{J^{-1}C}+o(1/n)$ by \cref{thm:Extension of Stone}. Hence, arguing similarly as for TE, we get $\E\CV = \E R(\thetahat) + o(1/n)$, ensuring that in expected value, CV is able to separate between estimators whose risk differ only in the $O(1/n)$-term. We stress that this holds in expected value only, and that the above shows nothing more than the fact that the bias of CV is of order $o(1/n)$. As noted before, the error of CV will be of order $\Opr(1/\sqrt{n})$ for any single sample, a term which is too large to properly capture effects of size $O(1/n)$.

The discussion given in the previous paragraphs has consequences for our theorems concerning the limiting behaviour of $\thetahat(\lambdahat)$. By the arguments above, neither TE nor CV is able to separate estimators whose risk differs only in the $O(1/n)$ term. Hence, when tuning $\lambda$ with respect to CV or TE, we should not expect that the criteria are able to identify the optimal value of $\lambda$ in this $O(1/n)$-regime. This also explains why tuning with respect to TE is asymptotically equivalent to tuning with respect to CV. 

The fact that the theory breaks down when $\theta_0(\lambda)$, the solution to $\E\varphi(Z,\theta,\lambda)=0$, is constant as a function of lambda, can also be seen from \cref{thm:alpha TE}. When $\theta_0(\lambda)$ is constant as a function of $\lambda$, the equation $\Psi(\alpha)=0$ has multiple solutions and $\Psi'(\alpha_0)$ will be singular. Because of this, the conditions of \cref{thm:alpha TE} do not hold true when $\theta_0(\lambda)$ is constant, making the theorem inapplicable in this case. This does not pose a problem for regularized estimators like in ridge regression, as $\theta_0(\lambda)$ will rarely be constant in this case, but for hybrid estimators, this can indeed be quite problematic. Take for instance the hybrid generative-discriminative classification model where $\thetahat(\lambda)$ is set to the maximizer of the following expression $n^{-1}\sum_{i=1}^n\crl{\lambda\log f_{Y\,|\,X}(Y_i,\theta\,|\,X=X_i) + (1-\lambda)\log f_{(Y,X)}(Y_i,X_i,\theta)}$ for some parametric model $f_{(Y,X)}$ for $(Y,X)$ and where $f_{(Y\,|\,X)}$ is derived from $f_{(Y,X)}$. In this case, $\thetahat(\lambda)$ aims for the maximizer of $\Gamma_\lambda(\theta) = \lambda\E\log f_{(Y,X)}(Y,X,\theta) + (1-\lambda)\E\log f_{Y\,|\,X}(Y,X,\theta)$. If the model is misspecified, $\theta_0(\lambda)$ will rarely be constant as a function of $\lambda$. If, on the other hand, the model does hold and the true underlying distribution is $f_{(Y,X)}(\cdot,\theta_0)$ for some $\theta_0\in\R^p$, $\theta_0(\lambda)$ will be constantly equal to $\theta_0$. Because of this, the difference in expected limiting risk will be on the order of $O(1/n)$ when the model is specified correctly and $O(1)$ otherwise. Hence, \cref{thm:thetahat truncated TE} and \cref{thm:thetahat CV asymptotically equivalent thetahat TE} can be applied in this case only when the model is misspecified. See Section S2 in the supplementary material for more discussions.

\section{When the effect of tuning disappears}\label{section: When phi = psi}

We will now go through a situation in which the expressions in \cref{cor:thetahat TE} simplify greatly and the effect of the tuning procedure becomes negligible asymptotically. 

Assume there exists a $\lambda_1$ such that $\varphi(y,\theta,\lambda_1) = \nabla_{\theta}\psi(y,\theta)$ for all $\theta$ and $F$-almost all $y$. Then, $\theta_0(\lambda_1)$ solves $\E\varphi(Z,\theta,\lambda_1)=\nabla_{\theta}\E\psi(Z,\theta) = 0$. Because of this, $\lambda_0$, $b$ and $Z_1$ defined in \cref{cor:thetahat TE} are equal to $\lambda_1$, $0$ and $J$, respectively. Entering this into the expressions in the corollary, shows that the right hand side of \eqref{eq:Influence of thetahat TE} is equal to $n^{-1}J^{-1}\varphi(z,\theta_0,\lambda_1)$. This is the influence function of $\thetahat(\lambda_1)$. Hence, $\sqrt{n}\p{\thetahat(\lambdahat)-\theta_0} = \sqrt{n}\crl{\thetahat(\lambda_1)-\theta_0}+\opr(1)$ in this case, and the additional randomness introduced by tuning $\lambda$ vanishes asymptotically. In practice, this means that one can ignore the effect of tuning $\lambda$ and use the classic approximation $\sqrt{n}\crl{\thetahat(\lambdahat)-\theta_0}\distconv\text{N}\crl{0,J^{-1}K\p{J^{-1}}^T}$, where $J = -\partial_{\theta_0}\E\varphi(Z,\theta,\lambda_1)$ and $K=\E\varphi(Z,\theta_0,\lambda_1)\varphi(Z,\theta_0,\lambda_1)^T$. 

A $\lambda_1$ such that $\nabla_\theta\psi(z,\theta)=\varphi(z,\theta,\lambda_1)$ actually exists in quite a few situations. In for instance ridge regression where $\lambda$ is set to the minimizer of estimated MSE, we have $\varphi(z,\beta,\lambda) = \nabla_{\beta}\crl{\p{y-\beta_{1:}^Tx - \beta_0}^2 + \lambda\norm{\beta_{1:}}^2}$, where $\beta =(\beta_0,\beta_{1:}^T)^T$, ensuring $\nabla_{\beta}\psi(z,\beta) = \varphi(z,\beta,0)$. More generally, for any model $Y_i = g(X_i,\theta) + \epsilon_i$ fitted by minimising $\sum_{i=1}^n\crl{Y_i-g\p{X_i,\theta}}^2 + \lambda\norm{\theta}^2$ and tuned according to the MSE, $\nabla_{\theta}\psi(z,\theta) = \varphi(z,\theta,0)$. Even more generally, for a model fitted by minimising $\sum_{i=1}^n h\crl{Y_i,g\p{X_i,\theta}} + \lambda\norm*{\theta}^2$ for some function $h$, and where $\lambda$ is tuned by minimising the CV or TE estimate of $\E h\sqb{Y,g\crl{X,\thetahat(\lambda)}}$, $\nabla_\theta\psi(z,\theta) = \varphi(z,\theta,0)$, ensuring that $\sqrt{n}\crl{\thetahat(\lambdahat)-\theta_0}$ and $\sqrt{n}\crl{\thetahat(0)-\theta_0}$ are asymptotically equivalent. This includes many regularized likelihood models when cross validated Kullback-Leibler divergence is used to tune $\lambda$.

The above arguments might lead one to believe that there rarely is need to use \cref{thm:thetahat truncated TE} or \cref{thm:thetahat CV asymptotically equivalent thetahat TE} in practice. This is, however, not correct. To illustrate the potentially large effect tuning of $\lambda$ can have on the limit distribution of $\thetahat(\lambdahat)$, we performed a small simulation experiment. For a selection of $\lambda$-values, we fitted logistic regression in two variables with a ridge penalty term. This corresponds to setting $\varphi$ equal to the gradient of $y\p{\beta_0 + \beta_1x_1 + \beta_2x_2} - \log\p{1+\beta_0 + \beta_1x_1 + \beta_2x_2} - \lambda\p{\beta_1^2 + \beta_2^2}$ with respect to $\beta = (\beta_0,\beta_1,\beta_2)$. We chose $\lambdahat$ by minimising CV$(\lambda)$ with $\psi(z,\theta) = \sqb{\exp\p{\beta_0 + \beta_1x_1}/\crl{1+\exp\p{\beta_0 + \beta_1x_1}} - y}^2$. For the true underlying distribution of the data, we chose the following. The binary variable $Y$ was set to be Bernoulli distributed with equal probability for zero and one. Furthermore, for a selection of $C$-values, we used $X\,|\,Y = 0\sim \text{N}\p{\mu,\Sigma}$ and $X\,|\,Y = 1\sim\text{N}\p{-\mu,2\Sigma}$ with $\mu = (1/2,1/2)^T$ and $\Sigma$ equal to the matrix with diagonal elements $2$ and off-diagonal entries $-\sqrt{C/2}$.

For each  selected $C$-value, we simulated $n=100$ data points and computed $\lambdahat_C$ and $\hat{\beta}_C(\lambdahat)$ based on the data set. We used cross-validation to tune $\lambda$. This procedure was repeated $B = 1000$ times for each $C$, and based on these $B$ samples, we computed the empirical variance of $\sqrt{n}\hat\beta_C(\lambdahat_C)$ for each $C$. We also estimated the variance using both the classic method and \cref{thm:thetahat CV asymptotically equivalent thetahat TE}. The results are summarized in \cref{Fig:Simmulation variances}, which shows the error of the two estimators as functions of $C$.  From \cref{Fig:Simmulation variances} it is clear that the classic estimator is insufficient in this case. The variance of \cref{thm:thetahat CV asymptotically equivalent thetahat TE} is closer to the observed variability in almost all cases. Furthermore, we see that as $C\to0$, the classic variance estimate decreases in effect. This is particularly clear for the estimate of the variance of the last component of $\sqrt{n}\hat{\beta}_C(\lambda_C)$, where the variance estimator provided by the classic method is a lot worse than the one proposed in this article and seems to increase even more as $C\to0$. This simulation is of course only a toy example, but \cref{Fig:Simmulation variances} indicates that the error introduced by using the classic estimator of the variance potentially can be very large.

\begin{figure}
	\centering
	\includegraphics[width=0.95\linewidth]{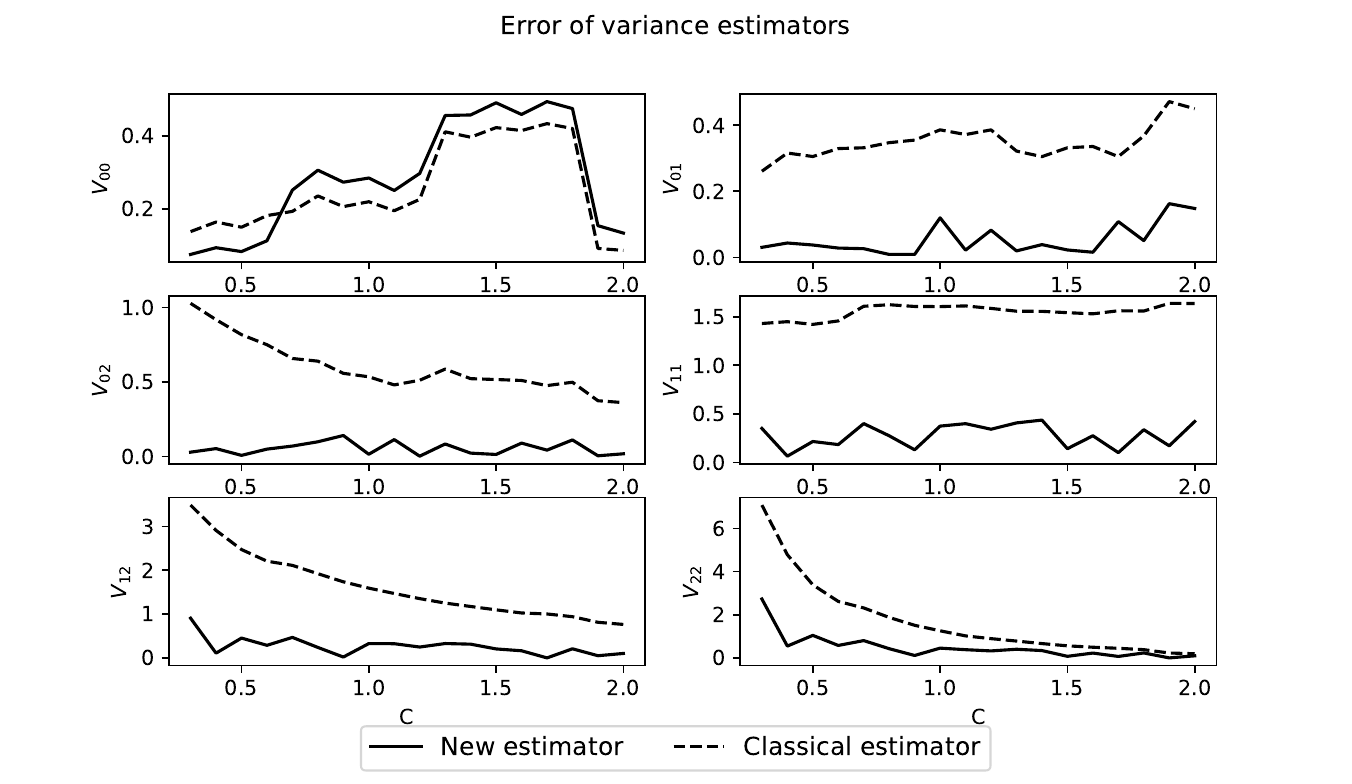}
	\caption{\label{Fig:Simmulation variances} The plot shows the absolute error of the classic and new estimated covariance of $\sqrt{n}\hat{\beta}_C(\lambdahat_C)$ compared to the observed variance based on a thousand simulated data sets with $n=100$ data points. The variables $V_{jk}$ refer to the covariance between the $j$-th and $k$-th component of $\sqrt{n}\betahat_C(\lambdahat_C)$ for $j=0,1,2$.}
\end{figure}

\section{Examples}\label{section:Illustrations}
We will now go through a couple of applications illustrating our theoretical framework.

\subsection{Diabetes in Pima Indians}\label{section:Pima indians}
The data set \texttt{pima.indians.diabetes}, publicly available in the R-package \texttt{mlbench} contains information about female Pima Indians and prevalence of diabetes. Denote the data set by $z_1,\ldots,z_n$ where $z_i = (y_i,x_i^T)^T$ with $x_i$ being a vector of covariates and $y_i$ a binary variable encoding whether or not the individual has diabetes for $i=1,\dots,n$. We fit a penalized logistic regression model to these data. For some pre-specified value of $\lambda$, this corresponds to maximising the following expression
\begin{equation}\label{eq:Likelihood in motivating example}
	\ell_n(\beta, \lambda) = \sum_{i=1}^n\sqb{y_i\log p\p{x_i,\beta} + (1-y_i)\log\crl{1-p(x_i,\beta)}} - n\lambda\norm{\beta_{1:}}^2
\end{equation} 
where $p(x,\beta) = \sqb{1+\exp\crl{-\beta_0 - \beta_{1:}^Tx}}^{-1}$ when $\beta = \p{\beta_0,\beta_{1:}^T}^T$. To fit this model, we need to decide upon a value for the tuning parameter $\lambda$. One option is to choose $\lambda$ minimising the Brier score defined as  BS $= n^{-1}\sum_{i = 1}^n\sqb{y_i-p\crl{x_i, \betahat(\lambda)}}^2$ where $\betahat(\lambda)$ is the minimizer of \eqref{eq:Likelihood in motivating example}. This ensures that the model has good predictive ability. In most cases, however, BS will underestimate the true error. To mitigate this, we will instead choose $\lambda$ minimising the cross-validated estimate of the Brier score, defined as $n^{-1}\sum_{i = 1}^n\sqb{1-p\crl{x_i, \betahat(\lambda)_{(-i)}}}^2$.

With the procedure described above, we find that the minimizer of the cross-validated Brier score is equal to $\lambdahat = 0.0085$. Fixing $\lambda$ at this value, we can now fit a penalized logistic regression model to the data by minimising \eqref{eq:Likelihood in motivating example} evaluated at $\lambda = \lambdahat$. Furthermore, by classic asymptotic likelihood theory, we know that for each fixed $\lambda$, $\betahat(\lambda)$ is consistent for some well-defined quantity $\beta(\lambda)$ and that $\sqrt{n}\crl{\betahat(\lambda) - \beta(\lambda)}$ converges in law to a N$\crl{0,J(\lambda)^{-1}K(\lambda)J(\lambda)^{-1}}$ distribution where $J(\lambda)$ and $K(\lambda)$ are well-defined matrices, see e.g.~\citet[Ch.~5]{van2000asymptotic}. By considering $\lambda$ as being fixed at $\lambdahat$, we can therefore approximate the limiting distribution of $\betahat(\lambdahat)$ by a N$\crl{\beta(\lambda),J(\lambdahat)^{-1}K(\lambdahat)J(\lambdahat)^{-1}/n}$ distribution, and use it to make inference about $\beta(\lambda)$. The resulting estimated marginal densities of each component of $\betahat(\lambdahat)$ are shown in \cref{fig:Motivating example old fits}. We remark that the very small value of $\lambdahat$ is a consequence of the penalization term being scaled by $n$, leading the ``actual'' size of the penalization to be $n\lambdahat \approx 3.342$. 

The above approximate distribution is based on pointwise results which do not take the randomness introduced by tuning $\lambda$ into account. This can, however, be done with the results of this paper. Assuming all regularity conditions are in place, we can apply the theory of this article to $\betahat(\lambdahat)$ by setting $\varphi(z,\beta,\lambda) = \nabla_{\beta} \sqb{y\log p\p{x,\beta}-(1-y)\log\crl{1 - p (x,\beta)} - \lambda\norm*{\beta_{1:}}^2}$ and $\psi(z,\beta) = \crl{1-p\p{x,\beta}}^2$. By \cref{thm:thetahat CV asymptotically equivalent thetahat TE}, $\sqrt{n}\crl{\betahat(\lambdahat)-\beta_0}$ converges in distribution to a central normal distribution with variance as described in \cref{cor:thetahat TE} and which can be estimated using the approximations of \cref{thm:Estimators of variance}. For each $j=0,\ldots,8$, densities of the marginal approximated distributions of $\betahat(\lambdahat)_j$ together with the densities obtained using the classic, pointwise results as well as histograms based on $B = 2000$ bootstrap samples of $\betahat(\lambdahat)_j$ can be found in \cref{fig:Motivating example old fits} which was discussed briefly in the introduction.

From \cref{fig:Motivating example old fits} we note that when taking the added randomness of the tuning procedure into account, we get larger estimated variances than when using the classic pointwise results. This is promising, as the estimated variance of $\betahat(\lambdahat)_j$  was closer to the empirical variance in the bootstrap sample for all but one value of $j$. For $\betahat(\lambdahat)_2$, the classic variance estimator performed better than what we get from \cref{thm:Estimators of variance}. A quick glance at \cref{fig:Motivating example old fits}, however, reveals that neither of the approximate distributions seem to fit well with the bootstrap distribution in this case. In fact, the histogram of this component is quite skewed, giving reason to believe that convergence has not been achieved for this particular parameter.

Although the results above are slightly in the favour of the theorems derived in the present article, the limiting distributions we get using our and the classic method are not astoundingly different. It is, however, important to remember that the classic result is a pointwise result concerning the limiting behaviour of $\betahat(\lambda)$ when $\lambda$ is a fixed number. Because of this, the classic result really does not apply to $\betahat(\lambdahat)$ and cannot be used to say anything about the limiting behaviour when $\lambdahat$ is set by tuning, as neither asymptotic normality nor consistency of $\betahat(\lambdahat)$ follows from standard theory. By using the classic pointwise method regardless of this fact, we are therefore really taking a gamble, hoping that the added randomness of the tuning procedure will be negligible. \cref{thm:thetahat CV asymptotically equivalent thetahat TE} and \cref{thm:Estimators of variance}, on the other hand, do take the tuning procedure into account and can safely be applied to make valid asymptotic inference.

\subsection{Hybrid likelihood}\label{section:Hybrid likelihood}
In the book \textit{The better angels of our nature}, \citet{pinker2012better} argues that the world is changing for the better and that violence is on the decline. This claim quickly stirred up heated debate in the academic community and since its publication a large number of articles have either attempted to support or refute Pinker's claim, see e.g.~\citet{cirillo2016decline, cirillo2016statistical} or \citet{cunen2020statistical}. Recently, \citet{daehlen2024hybrid} investigated whether conflicts have become less violent over time. We will now use the theory derived in the present article to improve the analysis done in that paper.

\citet{daehlen2024hybrid} uses data from the Correlates of War data base (\citet{CoW}) concerning the number of battle deaths in the 95 most recent and concluded inter-state wars. The analysis is done by first splitting the data into ``older'' and ``newer'' wars with the Korean war as the cut-off point. To each of these data sets, they fit shifted log-normal models and use the result to make inference about the difference in the median number of battle deaths for older and newer wars. They also repeat the analysis for the difference in upper quartiles. 

Further, to fit the parametric models, the framework of hybrid likelihood is used. We will not go into details concerning this method, but roughly speaking, they fit parametric models $f_\theta$ by minimising a convex combination of the log-likelihood function and a nonparametric counterpart called the log-empirical likelihood function, see e.g.~\citet{owen2001empirical}. How much weight is put on either of these components is decided by a parameter called the balance parameter $a$ taking values in $[0,1)$. 

Moreover, the FIC is used to select values for the balance parameters $a_1$ and $a_2$. When investigating change in the median number of battle deaths, the median is used as focus parameter and the empirical median as the consistent estimator $\hat\mu_0$. When studying the difference in upper quartile, the $0.75$-quantile and the empirical $0.75$-quantile are used in place of the median and empirical median. By standard results, the empirical $p$-quantile satisfies $\opr(1/\sqrt{n}) = n^{-1}\sum_{i=1}^n\xi(Y_i,\mu)$ where $\xi(y,\mu) = \mu -\mu_p - \crl{p-I(y\leq\mu_p)}/f(\mu_p)$ and $\mu_p$ is the $p$-quantile in the underlying distribution of the data. Furthermore, by the results of \citet{daehlen2024hybrid}, the estimators $\thetahat_1(a_1)$ and $\thetahat_2(a_2)$ can both be expressed as components of Z-estimators for each fixed value of $a_1$ and $a_2$. Here $\thetahat_1(a_1)$ and $\thetahat_2(a_2)$ are the maximum hybrid likelihood estimators for the older and newer wars respectively using the balancing parameter given in the parenthesis. Hence, by making modifications as explained in \cref{section:Information criteria}, we can use \cref{cor:thetahat TE} to derive the limiting distribution of $\thetahat_1(\hat a_1)$ and $\thetahat_2(\hat a_2)$.

By combining the above with the estimators given in \cref{thm:Estimators of variance}, we can approximate the variance of $g\crl{\thetahat(\hat a_1)} - g\crl{\thetahat(\hat a_2)}$, where $g(\theta)$ is the median in a log-normal distribution parametrized by $\theta$, in a way taking the added randomness introduced by tuning $a_1$ and $a_2$ into account. Using this method, we get an estimated standard deviation of 18431. That is a lot higher than 3815, which is obtained using the pointwise results in \citet{daehlen2024hybrid}. This is a consequence of \citet{daehlen2024hybrid} not incorporating the added randomness introduced by tuning $a_1$ and $a_2$, leading to a much too small estimated standard deviation.

For the difference in upper quartiles, $\hat a_1 = 0$ and $\hat a_2$ is chosen as the largest accepted value of $a_2$ (in this case $0.99$). Because of this, FIC$_j'(\hat a_j)\neq 0$ for $j =0,1$ and \cref{thm:thetahat truncated TE} leads us to approximate the limiting distribution of $\thetahat_j(\hat a_j)$ by the pointwise limit of $\thetahat_j(\hat a_j)$ for $j = 0,1$. This leads to the same approximate distribution as what is used in \citet{daehlen2024hybrid}, and hence, all estimated standard deviations agree. As a result, ignoring the effect of tuning of $a_1$ and $a_2$  becomes unproblematic in this case.

\section{Concluding remarks}\label{section:Concluding remarks}
We have studied the effect tuning procedures has on the limiting distribution of $\thetahat(\lambdahat)$. We covered multiple ways of tuning, including minimization of information criteria, the training error and the cross-validated estimator of the risk. In addition, we have defined consistent estimators for the limiting variance of $\sqrt{n}\thetahat(\lambdahat)$ and proved a result sharply characterizing the distance from CV to TE. Lastly, we went through a simulation setting and applied the theory on two real data sets.

Our work does have some limitations. The perhaps most pressing issue with our current theory is the inability to handle none-smooth functions $\varphi$ and $\psi$. This excludes multiple interesting settings. Since, $\psi$ must be smoothly differentiable, tuning with respect to the absolute error or error rate is not covered by our theory. Furthermore, smoothness of $\varphi$ excludes many interesting models, like the Lasso or other models penalized by an $L_1$-norm. This does, of course, make our theory less applicable, but the smoothness assumptions are crucial for our proofs, and we believe that extending our theory to cover non-smooth settings is far from trivial.

Further, we have assumed that for each $\lambda$, $\thetahat(\lambda)$ is a Z-estimator. This assumption is necessary for the proof of \cref{thm:thetahat CV asymptotically equivalent thetahat TE} and \cref{thm:Extension of Stone}, but might be more strict than what is really needed at least for the latter theorem. By the informal arguments preceding \cref{thm:Extension of Stone}, one would expect that \cref{thm:thetahat CV asymptotically equivalent thetahat TE} and \cref{thm:Extension of Stone} hold true as long as $\thetahat(\lambda)$ has an influence function IF, such that $\thetahat(\lambda) = \theta_0(\lambda) + \textup{IF}(Y_i,\theta,\lambda) + \epsilon_n(\theta,\lambda)$ with ``regular enough'' remainder term. Formalizing this intuition in a proper way would certainly extend our framework, but will require proving a more general version of Lemma S1.3 in the supplementary material.

Finally, we have focused on \cref{thm:thetahat truncated TE} and \cref{thm:thetahat CV asymptotically equivalent thetahat TE} as our goal was to study the limiting behaviour of $\thetahat(\lambdahat)$ when $\lambdahat$ is set by some tuning procedure. These results are, however, corollaries of \cref{thm:alpha TE}, which gives the limiting distribution of the full vector $(\thetahat(\lambdahat)^T, \lambdahat^T, \thetahat'(\lambdahat)^T)^T$. In particular, \cref{thm:alpha TE} can be used to study the limiting behaviour of $\lambdahat$. This can, in turn, be applied to make inference about $\lambda$ and create approximate confidence intervals and hypothesis tests for the tuning parameter. For instance, it might be used to formally test if a regularization is beneficial or if either of the two components in a hybrid model is significantly preferable.

\section*{Supplementary Materials}
Proofs, lemmas and further discussions can be found in the article's supplementary material.

\section*{Acknowledgments}
This work is partially funded by the Norwegian Research council through the BigInsight center for innovation driven research (237718). Partial support was also given from Centre for Advanced Study, at the Academy of Science and Letters, Oslo, in connection with the 2022-2023 project Stability and Change, led by H. Hegre and N.L. %The authors are indebted to insightful input from our reviewers, which helped improve our manuscript.

\bibliography{references}

\end{document}

\centerline{\large\bf The asymptotic effect of tuning parameters - Supplementary material}
\vspace{.4cm} 
\centerline{Ingrid Dæhlen$^{1,2}$, Nils Lid Hjort$^1$ and Ingrid Hobæk Haff$^1$} 
\vspace{.4cm} 
\centerline{$^1$ Department of Mathematics, University of Oslo}
\centerline{$^2$ Norwegian Computing Center, Post box 114 Blindern, Oslo, 0314, Norway}
\vspace{.4cm} 
Address for correspondence: Ingrid Dæhlen, Department of Mathematics, University of Oslo, Moltke Moes vei 35, 0851 Oslo, Noway, email: ingrdae@math.uio.no

\section{Proof of Lemma 1 and Theorem 6}
We start by showing a lemma regarding uniform consistency.
\begin{lemma}\label{lemma:Uniform consistency}
	Let $Y\in\R^d$, $\theta\in\R^p$ and $\lambda\in\Lambda\subseteq\R^q$ where $\Lambda$ is a compact set. With $\Psi_{\lambda,n}\colon\R^{d+p}\to\R^p$ and $\Psi_\lambda\colon\R^p\to\R^p$ for each $\lambda\in\Lambda$, let $\theta_0(\lambda)$ be the unique zero of $\Psi_\lambda(\theta)$ for each fixed $\lambda$ and assume the following.
	\begin{itemize}
		\item [(C1)]
		There exists $\epsilon>0$ such that $\Psi_{\lambda,n}\prconv\Psi_\lambda$ uniformly on $\mathcal{S}_\epsilon = \crl{(\theta,\lambda)\,|\,\lambda\in\Lambda,\, \norm{\theta_0(\lambda)-\theta}\leq\epsilon}$ 
		
		\item[(C2)]
		The function $(\theta,\lambda)\mapsto\Psi_\lambda(\theta)$ is continuously differentiable on $\mathcal{S}_\epsilon$ with $\Psi_\lambda'\{\theta(\lambda)\}$ non-singular for each $\lambda\in\Lambda$.
	\end{itemize}
	Then, if $\Psi_{\lambda,n}(\theta)$ has at most a single zero $\thetahat(\lambda)$ for each $n$ and each $\lambda$, 
	\begin{equation}\label{eq:Consistency equation in appendix}
		\sup_\lambda\norm*{\thetahat(\lambda)-\theta_0(\lambda)}\prconv0.
	\end{equation}
\end{lemma}
\begin{proof}
	The following is a modified version of the proof of  theorem 5.42 in \citet{van2000asymptotic}.	
	
	By (C2) and the implicit function theorem, $\theta(\lambda)$ is continuous as a function of $\lambda$ in $\Lambda$. Hence, $\crl{\theta(\lambda)\,|\,\lambda\in\Lambda}$ is a compact set. This implies that $\mathcal{S}_\epsilon$ is compact. Because of this, the extreme value theorem and the fact that $\Psi'_\lambda\{\theta(\lambda)\}$ is non-singular for all $\lambda\in\Lambda$, we can choose $\epsilon$ to be small enough that $\Psi_\lambda'(\theta)$ is non-singular for all $(\theta,\lambda)\in\mathcal{S}_\epsilon$, ensuring that $(\theta,\lambda)\mapsto\Psi_\lambda'(\theta)^{-1}$ is continuous on $\mathcal{S}_\epsilon$. By the extreme value theorem, we then have that the norm of each component of $\Psi_\lambda'(\theta)^{-1}$ is bounded by some $C<\infty$ in $\mathcal{S}_\epsilon$. 
	
	Now fix some $\delta>0$ and choose $\delta_1$ smaller than both $\epsilon$ and $\delta/2C$. Arguing as in \citet{van2000asymptotic}, we see that there exists closed neighbourhoods $G^\lambda_{\delta_1}$ of $\theta_0(\lambda)$ such that $\Psi^\lambda$ is a homeomorphism from $G^\lambda_{\delta_1}$ into the closed ball around zero with radius $\delta_1$: $\overline{B_{\delta_1}(0)}$ for each $\lambda$. By construction of $\delta_1$ and the inverse function theorem, the partial derivatives of the inverse of $\Psi_\lambda$ is bounded by $C$ for each $\lambda$. Hence, the diameter of $G_{\delta_1}^\lambda$ is bounded by $2C\delta_1<\delta$ for each $\lambda$ implying $G_{\delta_1}^\lambda\subseteq \overline{B_{\delta}\crl{\theta_0(\lambda)}}$ for each $\lambda$. We will show that $\Psi_{\lambda,n}(\theta)$ has a root in $G_{\delta_1}^\lambda$ for every $\lambda$ with probability tending to one. Since $G_{\delta_1}^\lambda\subseteq \overline{B_{\delta}\crl{\theta_0(\lambda)}}$ for each $\lambda$, this would guarantee \eqref{eq:Consistency equation in appendix}.
	
	Since $\delta_1<\epsilon$, we know
	\begin{equation*}
		\sup_{\lambda\in\Lambda}\sup_{\theta\in G_{\delta_1}^\lambda}\norm*{\Psi_{\lambda,n}(\theta)-\Psi_\lambda(\theta)}\prconv0.
	\end{equation*}
	Let $K_{\delta_1,n}$ be the event that $\norm*{\Psi_{\lambda,n}(\theta)-\Psi_\lambda(\theta)}<\delta_1$ for all $\theta\in G_{\delta_1}^\lambda$ and $\lambda\in\Lambda$. Then the above implies that $\Pr\p{K_{\delta_1,n}}\to1$.
	
	Now let $y\in \overline{B_{\delta_1}(0)}$. Since $\Psi_\lambda$ is a homeomorphism from $G_{\delta_1}^\lambda$ into $\overline{B_{\delta_1}(0)}$ for each $\lambda$, $\Psi_\lambda^{-1}(y)\in G_{\delta_1}^\lambda$ for every $\lambda\in\Lambda$. Hence, on the event $K_{\delta_1,n}$, we have $\norm{\Psi_{\lambda,n}\crl{\Psi_\lambda^{-1}(y)} - y}\leq \delta_1$ for all $\lambda\in\Lambda$, which implies that $y\mapsto y-\Psi_{\lambda,n}\circ\Psi_\lambda^{-1}(y)$ maps $\overline{B_{\delta_1}(0)}$ into itself for every $\lambda\in\Lambda$. Since the maps are continuous, Brouwer's fixed point theorem ensures that $y\mapsto y-\Psi_{\lambda,n}\circ\Psi_\lambda^{-1}(y)$ has a fixed point in $\overline{B_{\delta_1}(0)}$ for every $\lambda\in\Lambda$. This is equivalent to $\Psi_{\lambda,n}$ having a zero in $G_{\delta_1}^\lambda\subseteq \overline{B_{\delta}\crl{\theta_0(\lambda)}}$ for each $\lambda$. Because of this, the probability of all $\Psi_{n,\lambda}$ having a zero within $\delta$ of $\theta_0(\lambda)$ is bounded by $\Pr(K_{\delta_1,n})$. As this latter probability tends to one by the previous arguments, \eqref{eq:Consistency equation in appendix} follows. This concludes the proof.
\end{proof}

When $\Psi_\lambda(\theta) = \E \phi(Z,\theta,\lambda)$ and $\Psi_{\lambda,n}(\theta) = n^{-1}\sum_{i=1}^n\phi(Z_i,\theta,\lambda)$  for some function $\phi\colon\R^{d+p+q}\to\R^p$, condition (C1) can be easily verified using the uniform law of large numbers. We state this in the following corollary.

\begin{corollary}\label{corollary:Uniform consistency}
	Let $\Psi_\lambda(\theta) = \E \phi(Z,\theta,\lambda)$ and $\Psi_{\lambda,n}(\theta,\lambda) = n^{-1}\sum_{i=1}^n\phi(Z_i,\theta,\lambda)$  for some function $\phi\colon\R^{d+p+q}\to\R^p$. Then (C1) follows from 
	\begin{itemize}
		\item[(C3)]
		There exists an $F$-integrable function $m_0\colon\R^d\to\R$ and an $\epsilon>0$ such that $\norm{\phi(z,\theta,\lambda)}\leq m_0(z)$ for all $(\theta,\lambda)\in S_{\epsilon}$ and $F$-almost all $Z$.
	\end{itemize}
\end{corollary}

We will prove a theorem characterizing the distance from $\thetahat(\lambda)$ to $\thetahat\noti(\lambda)$.

\begin{lemma}\label{lemma:appendix lemma for alternative proof}
	For each $\lambda$ in $\Lambda$, a compact subset of $\R^q$ and $1\leq i\leq n$ let $\thetahat(\lambda)$ and $\thetahat_{(-i)}(\lambda)$ be the solutions to $\sum_{j=1}^n\phi(Z_j,\theta,\lambda)$ and $\sum_{j\neq i}\phi(Z_j,\theta,\lambda)$ in $\Theta$, a compact subset of $\R^p$. Assume there exist functions $m_0$ and $m_1$ such that $\norm{\phi\p{z,\theta,\lambda}}\leq m_0(z)$ and $\norm{\partial_\theta\phi\crl{z,\theta,\lambda}}\leq m_1(z)$ for all $\theta, \lambda$ and almost all $z$. Then, if $\theta\mapsto\E\phi(Z,\theta,\lambda)$ satisfies (C2) of \cref{lemma:Uniform consistency} and the eigenvalues of $J(\theta,\lambda)^{-1} = -\crl{\E\partial_\theta\phi(Z,\theta,\lambda)}^{-1}$ are bounded in $\Theta\times\Lambda$, the following holds for all $\lambda$:
	\begin{equation}\label{eq:First bound in lemma for thetahat-thetahati}
		\norm{\thetahat_{(-i)}(\lambda)-\thetahat(\lambda)}\leq a_{0,n}m_0(Z_i)
	\end{equation}
	where $a_{0,n} = \Opr(1/n)$ and does not depend on $i$ or $\lambda$. If, furthermore, the eigenvalues of $\partial_\theta J(\theta,\lambda)$ are bounded in $\Theta\times\Lambda$, we have
	\begin{equation}\label{eq:Second bound in lemma for thetahat-thetahati}
		\thetahat_{(-i)}(\lambda) - \thetahat(\lambda) = n^{-1}J\crl{\theta_0(\lambda),\lambda}^{-1}\phi\crl{Z_i,\theta_0(\lambda),\lambda} + E_n^i(\lambda)
	\end{equation}
	where $E_n^i(\lambda)$ satisfies
	\begin{equation}
		\norm{E_n^i(\lambda)} \leq n^{-1}\crl{a_{1,n} + a_{2,n}m_0(Z_i) + a_{3,n}m_0(Z_i)^2 + a_{4,n}m_0(Z_i)^2m_1(Z_i)}
	\end{equation}
	with $a_{j,n} = \opr(1)$ uniformly in $\lambda$, not depending on $i$ for $j = 1,2,3,4$.
\end{lemma}
\begin{proof}
	By definition we have
	\begin{equation}
		0 = \sum_{j=1}^n\phi\crl{Z_j,\thetahat(\lambda),\lambda}\hspace{3mm}\text{and}\hspace{3mm}0 = \sum_{j\neq i}^n\phi\crl{Z_j,\thetahat_{(-i)}(\lambda),\lambda}.
	\end{equation}
	Hence, a Taylor expansion of $\sum_{j=1}^n\phi\crl{Z_j,\thetahat\noti(\lambda),\lambda}$ around $\thetahat(\lambda)$ reveals
	\begin{align*}
		\phi\crl{Z_i,\thetahat_{(-i)}(\lambda),\lambda} &= \sum_{j=1}^n\phi\crl{Z_j,\thetahat_{(-i)}(\lambda),\lambda} \\
		&= \sum_{j=1}^n\phi\crl{Z_j,\thetahat(\lambda),\lambda} + \p*{\sum_{j=1}^n\partial_{\theta_i^*(\lambda)}\phi\crl{Z_j,\theta,\lambda}}\crl{\thetahat_{(-i)}(\lambda)-\thetahat(\lambda)}\\
		&=-nJ_n\crl{\theta_i^*(\lambda),\lambda}\crl{\thetahat_{(-i)}(\lambda)-\thetahat(\lambda)}
	\end{align*}
	for some $\theta_i^*(\lambda)$s on the line segment between $\thetahat_{(-i)}(\lambda)$ and $\thetahat(\lambda)$, and with $J_n\crl{\theta,\lambda} = -n^{-1}\sum_{j=1}^n\partial_{\theta}\phi(Z_j,\theta,\lambda)$. Because of the above, we can write 
	\begin{equation*}
		\thetahat_{(-i)}(\lambda)-\thetahat(\lambda) = -n^{-1}J_n\crl{\theta_i^*(\lambda),\lambda}^{-1}\phi\crl{Z_i,\thetahat_{(-i)}(\lambda),\lambda}.
	\end{equation*}
	
	Since there exists a function $m_1(z)$ bounding all partial derivatives of $\phi(z,\theta,\lambda)$ with respect to $\theta$, $J_n(\theta,\lambda)$ converges uniformly to $J(\theta,\lambda) = -\E\partial_{\theta}\phi(Z,\theta,\lambda)$ over the compact parameter space $\Theta\times\Lambda$. Furthermore matrix inversion is uniformly continuous over compact sets not containing non-invertible matrices, ensuring that also $-J_n(\theta,\lambda)^{-1}$ converges to $-J(\theta,\lambda)^{-1} = \E\partial_{\theta}\phi(Z,\theta,\lambda)^{-1}$ uniformly in $\lambda$ and $\theta$. Then $-J_n\crl{\theta_i^*(\lambda),\lambda}^{-1} = -J\crl{\theta_i^*(\lambda),\lambda}^{-1} + \epsilon_n^i$ where $\epsilon_n^i = \opr(1)$ uniformly in $\lambda$ and $i$. By assumption, the eigenvalues of $J(\theta,\lambda)^{-1}$ are bounded by a number with absolute value $C_1$. Then  $\norm{\thetahat_{(-i)}(\lambda)-\thetahat(\lambda)} \leq n^{-1}\p{C_1 + \delta_n}\norm{\phi\crl{Z_i,\thetahat_{(-i)}(\lambda),\lambda}}$, where $\delta_n=\opr(1)$ does not depend on $\lambda$ or $i$. Lastly by assumption $\norm{\phi(z,\theta,\lambda)}\leq m_0(z)$. The above then guarantees
	\begin{equation}\label{eq:Bound for thetahati - thetahat first order}
		\norm*{\thetahat_{(-i)}(\lambda) -\thetahat\p{\lambda}} \leq n^{-1}\p{C_1+\delta_n}m_0(Z_i)
	\end{equation}
	for all $\lambda$. This proves \eqref{eq:First bound in lemma for thetahat-thetahati}.
	
	By the previous arguments
	\begin{equation}\label{eq:Second order bound for thetai - theta first}
		\thetahat_{(-i)}(\lambda)-\thetahat(\lambda) = -n^{-1}\sqb{J\crl{\theta_i^*(\lambda),\lambda}^{-1} + \epsilon_n^i}\phi\crl{Z_i,\thetahat_{(-i)}(\lambda),\lambda}.
	\end{equation}
	where $\epsilon_n^i = \opr(1)$ uniformly in $i$ and $\lambda$.	We will again use Taylor expansions to arrive at a result. We start with the second factor on the right hand side: $J\crl{\theta_i^*(\lambda),\lambda}^{-1} = J\crl{\theta_0(\lambda),\lambda}^{-1} + \p*{\partial_{\theta_i^{**}(\lambda)}J\crl{\theta,\lambda}^{-1}}\crl{\theta_i^*(\lambda) - \theta_0(\lambda)}$, where $\theta_i^{**}(\lambda)$ is some vector on the line segment between $\theta_0(\lambda)$ and $\theta_i^*(\lambda)$. Now since $\theta_i^*(\lambda)$ lies on the line segment between $\thetahat(\lambda)$ and $\thetahat_{(-i)}(\lambda)$, $\norm{\theta_i^*(\lambda)-\theta_0(\lambda)}\leq \norm{\thetahat(\lambda)-\theta_0(\lambda)} + \norm{\thetahat_{(-i)}(\lambda)-\thetahat(\lambda)}$. Combining this with \cref{lemma:Uniform consistency} and \eqref{eq:Bound for thetahati - thetahat first order} ensures $\norm{\theta_i^*(\lambda)-\theta_0(\lambda)} \leq R_n + n^{-1}\p{C_1+\delta_n}m_0(Z_i)$, where $R_n = \opr(1)$  uniformly in  $\lambda$ and does not depend on $i$. In addition, the eigenvalues of $\partial_\theta J(\theta, \lambda)^{-1}$ are bounded by $C_2$. Hence, $J\crl{\theta_i^*(\lambda),\lambda}^{-1} = J\crl{\theta_0(\lambda),\lambda}^{-1} + r_n^i(\lambda)$ where $\norm{r_n^i(\lambda)} \leq n^{-1}(C_1+\delta_n)C_2m_0(Z_i)$ uniformly in $\lambda$.
	
	Secondly, we will work with the third factor on the right hand side of \eqref{eq:Second order bound for thetai - theta first}. Taylor expanding the function $\phi\crl{Z_i,\thetahat\noti(\lambda),\lambda}$ around $\theta_0(\lambda)$, reveals
	\begin{equation*}
		\phi\crl{Z_i,\thetahat_{(-i)}(\lambda),\lambda} = \phi\crl{Z_i,\theta_0(\lambda),\lambda} + \partial_{\theta_i^{***}}\phi(Z_i,\theta,\lambda)\crl{\thetahat_{(-i)}(\lambda)-\theta_0(\lambda)},
	\end{equation*}
	for some $\theta_i^{***}(\lambda)$ on the line segment between $\theta_{(-i)}(\lambda)$ and $\theta_0(\lambda)$. Since there exists $m_1(z)$ bounding the norm of each component of $\partial_\theta\phi(z,\theta,\lambda)$ and depending only on $z$, we can use arguments similar to those above to show $\phi\crl{Z_i,\thetahat_{(-i)}(\lambda),\lambda} = \phi\crl{Z_i,\theta_0(\lambda),\lambda} + s_n^i(\lambda)$ where $\norm{s_n^i(\lambda)} \leq n^{-1}(C_1 + \delta_n)m_1(Z_i)m_0(Z_i)$ uniformly in $\lambda$ and for all $i$.
	
	We are now ready to complete the proof. Combining all of the previous arguments shows
	\begin{equation*}
		\thetahat_{(-i)}(\lambda)-\thetahat(\lambda) = -n^{-1}\sqb{J\crl{\theta_0(\lambda),\lambda}^{-1} + r_n^i(\lambda) + \epsilon_n^i}\sqb{\phi\crl{Z_i,\theta_0(\lambda),\lambda} + s_n^i(\lambda)},
	\end{equation*}
	which imply $\thetahat_{(-i)}(\lambda)-\thetahat(\lambda) = -n^{-1}J\crl{\theta_0(\lambda),\lambda}\phi\crl{Z_i,\theta_0(\lambda),\lambda} + E_n^i(\lambda)$ where
	\begin{equation*}
		\norm{E_n^i(\lambda)} \leq n^{-1}\crl{a_{1,n} + a_{2,n}m_0(Z_i) + a_{3,n}m_0(Z_i)^2 + a_{4,n}m_0(Z_i)^2m_1(Z_i)}
	\end{equation*}
	and $a_{j,n} = \opr(1)$ uniformly in $\lambda$ and does not depend on $i$ for $j = 1,2,3,4$. This concludes the proof.
	
\end{proof}

The first part of \cref{lemma:appendix lemma for alternative proof} shows Lemma 1 in the article. We will now use the second part to show Theorem 6.

\subsection{Proof of Theorem 6}
	Taylor's theorem ensures, $\sum_{i=1}^n\psi\p{Z_i,\thetahat\noti} = \sum_{i=1}^n\psi\p{Z_i,\thetahat} + \sum_{i=1}^n\nabla_{\theta_i^*}\psi(Z_i,\theta)^T\p{\thetahat\noti-\thetahat}$ for some $\theta_i^*$s on the line segments between $\thetahat$ and $\thetahat\noti$. By assumption, all conditions of \cref{lemma:appendix lemma for alternative proof} hold true, and hence,
	\begin{equation}\label{eq:Stone proof 1}
		\begin{aligned}
			\sum_{i=1}^n\nabla_{\theta_i^*}\psi(Z_i,\theta)^T\p{\thetahat\noti-\thetahat} =
			-n^{-1}\sum_{i=1}^n\nabla_{\theta_i^*}\psi(Z_i,\theta)^TJ^{-1}\varphi(Z_i,\theta_0) + \sum_{i=1}^n\nabla_{\theta_i^*}\psi(Z_i,\theta)^TE_N^i.
		\end{aligned}
	\end{equation}
	By the same lemma and existence of $p_1$, we have
	\begin{align*}
		\norm*{\sum_{i=1}^n\nabla_{\theta_i^*}\psi(Z_i,\theta)^TE_N^i}&\leq
		a_{1,n}n^{-1}\sum_{i=1}^np_1(Z_i) + a_{2,n}n^{-1}\sum_{i=1}^nm_0(Z_i)p_1(Z_i) \\
		&+ a_{3,n}n^{-1}\sum_{i=1}^nm_0(Z_i)^2p_1(Z_i) + a_{4,n}n^{-1}\sum_{i=1}^nm_0(Z_i)^2m_1(Z_i)p_1(Z_i).
	\end{align*}
	Since the moments of all terms in the above equation exists by assumption and $a_{j,n}=\opr(1)$ for $j=1,2,3,4$, the second term on the right hand side of \eqref{eq:Stone proof 1} is $\opr(1)$.
	
	For the first term of \eqref{eq:Stone proof 1}, note that by Taylor's theorem $\nabla_{\theta_i^*}\psi(Z_i,\theta) = \nabla_{\theta_0}\psi(Z_i,\theta) + H_{\theta_i^{**}}\psi(Z_i,\theta)(\theta_i^{*}-\theta_0)$ for some $\theta_i^{**}$ on the line segment between $\theta_i^*$ and $\theta_0$. By the triangle inequality and the fact that $\theta_i^*$ lies on the line segment between $\thetahat\noti$ and $\thetahat$, $\norm{\theta_i^* - \theta_0}\leq\norm{\thetahat-\theta_0} + \norm{\thetahat-\thetahat\noti}$. Hence, by \cref{lemma:appendix lemma for alternative proof}, we have
	\begin{align*}
		&\norm*{-n^{-1}\sum_{i=1}^n\nabla_{\theta_i^*}\psi(Z_i,\theta)^TJ(\theta_0)^{-1}\varphi(Z_i,\theta_0) - n^{-1}\sum_{i=1}^n\nabla_{\theta_0}\psi(Z_i,\theta)^TJ(\theta_0)^{-1}\varphi(Z_i,\theta_0)}\leq\\
		&\norm{\thetahat-\theta_0}n^{-1}C\sum_{i=1}^np_2(Z_i)m_0(Z_i) +a_{0,n}n^{-1}C\sum_{i=1}^np_2(Z_i)m_0(Z_i)^2,
	\end{align*}
	for some $C$. This expression is $\opr(1)$ since the expected value of both means are finite and $\norm{\thetahat-\theta_0},\,a_{0,n}=\opr(1)$ by \cref{lemma:appendix lemma for alternative proof}. This concludes the proof.

\section{Appendix - When the limit function is flat}\label{section:Flat grense}
In Theorem 1, we assumed that $\lambda_0$ was the unique minimizer of the limiting risk $R(\lambda) = \E\psi\crl{Z,\theta_0(\lambda)}$. In certain cases, however, $\theta_0(\lambda)$, the solution to $\E\varphi\crl{Z,\theta,\lambda} = 0$, is constant as a function of $\lambda$, and as a result, $R(\lambda)$ is flat. This is the case for hybrid models on the form $\varphi(z,\theta,\lambda) = \lambda\varphi_1(z,\theta) + (1-\lambda)\varphi_2(z,\theta)$ when  $\E\varphi_1(Z,\theta)$ and $\E\varphi_2(Z,\theta)$ share roots, e.g.~in the hybrid generative-discriminative model when the generative model is specified correctly or for hybrid combinations of empirical and hybrid likelihood functions (see \citet{daehlen2024hybrid}) under model conditions. In fact, when $R(\lambda)$ is flat, the limit of both TE$(\lambda)$ and CV$(\lambda)$ are constant, and there is no obvious ``best'' value for $\lambda$, making it unclear what $\lambdahat$ even should be aiming for were it consistent. We will now use informal mathematical arguments and simulations to illustrate why a clear and informative limit result for $\thetahat(\lambdahat)$ is hard to find when $\theta_0(\lambda)$ is constant.

Assume that $\theta_0(\lambda)$ is constant and equal to $\theta_0$. To investigate the behaviour of $\sqrt{n}\crl{\thetahat(\lambdahat)-\theta_0}$ we will first take a closer look at the function CV$(\lambda)$. By Theorem 6, the following holds pointwise,
\begin{equation*}
	\textup{CV}(\lambda) = n^{-1}\sum_{i=1}^n\psi\crl{Z_i, \thetahat(\lambda), \lambda} - n^{-1}\Tr\crl{J(\theta_0,\lambda)^{-1}V(\theta_0,\lambda)} + \opr(1/n).
\end{equation*}
Taylor expanding the first term on the right hand side in the above equation, reveals
\begin{align*}
	n^{-1}\sum_{i=1}^n\psi\crl{Z_i, \thetahat(\lambda), \lambda} =& n^{-1}\sum_{i=1}^n\psi\p{Z_i,\theta_0,\lambda} + \crl{\thetahat(\lambda)-\theta_0}^Tn^{-1}\sum_{i=1}^n\nabla_{\theta_0}\psi\p{Z_i,\theta,\lambda}\\
	&+\crl{\thetahat(\lambda)-\theta_0}^T\frac{1}{2n}\sum_{i=1}^nH_{\theta^*}\psi\p{Z_i,\theta,\lambda}\crl{\thetahat(\lambda)-\theta_0},\\
\end{align*}
for some $\theta^*$ on the line segment between $\thetahat(\lambda)$ and $\theta_0$. Assume $\theta_0$ is the true minimizer of $\E\psi(Z,\theta)$, i.e.~that $\E\nabla_{\theta_0}\psi\p{Z,\theta}=0$. Then, arguments similar to those given in the proof of Lemma 4.2  and Theorem 6 show
\begin{align*}
	S_n(\lambda) &= n\p*{\textup{CV}(\lambda)-\frac{1}{n}\sum_{i=1}^n\psi(Z_i,\theta_0)} \\
	&= U_n(\lambda)^TV_n - \Tr\crl{J(\theta_0,\lambda)^{-1}V(\theta_0)} + \frac{1}{2}U_n(\lambda)^TW_nU_n(\lambda)+\opr(1),
\end{align*}
where $U_n(\lambda) = \sqrt{n}\crl{\thetahat(\lambda)-\theta_0}$, $V_n = n^{-1}\sum_{i=1}^n\nabla_{\theta_0}\psi(Z_i,\theta)$, and $W_n=n^{-1}\sum_{i=1}^nH_{\theta_0}\psi(Z_i,\theta)$. Since $\theta_0$ is constant with respect to $\lambda$, we have $\lambdahat = \text{argmin}_\lambda S_n(\lambda)$.

Note that $W_n$ is consistent for $W = \E H_{\theta_0}\psi(Z,\theta)$ by the assumptions of Theorem 6. Furthermore, it  does not depend on $\lambda$. If, in addition, the conditions of \citet[Ex.~19.7]{van2000asymptotic} are fulfilled for $(z,\lambda)\mapsto\p{J\p{\theta_0,\lambda}^{-1}\varphi\p{z,\theta_0,\lambda},\psi\p{z,\theta_0}}$, the process $\lambda\mapsto\p{U_n\p{\theta},V_n}$ converges as a process to a Gaussian process $(U(\lambda),V)$ with covariance structure as described in the beginning of \citet[Sec.~19.2]{van2000asymptotic}. This, in combination with Slutsky's lemma and the continuous mapping theorem, is sufficient for $S_n$ to converge in distribution in Skorokhod space to the following limit process
\begin{equation}\label{eq:Limit process within model conditions}
	S(\lambda) = U(\lambda)^TV - T(\lambda) + \frac{1}{2}U(\lambda)^TWU(\lambda),
\end{equation}
where $T(\lambda) = \Tr\crl{J(\theta_0,\lambda)^{-1}V(\theta_0)}$.

Direct computations show $\E S(\lambda) = (1/2)\Tr\crl{WJ(\lambda)^{-1}K(\lambda)J(\lambda)^{-1}}$. Hence, if there exists a $\lambda_0$ such that $\thetahat(\lambda_0)$ is equal to the maximum likelihood estimator, the Cramer-Rao theorem guarantees that $\lambda_0$ minimizes $\E S(\lambda)$. Since such a $\lambda_0$ often exists in the case where $\theta_0(\lambda)$ is constant, it is tempting to think that $\lambdahat$ converges in probability to this $\lambda_0$. Sadly this conclusion cannot be drawn, as $S$ is a random process and there is no guarantee that minimizing the expected value of $S$ minimizes the process itself. Furthermore, the convergence of $S_n$ towards $S$ is in distribution. Because of this, convergence in probability of $\lambdahat$ towards $\lambda_L$, the minimizer of $S$, does not follow. We can only deduce $\lambdahat\distconv\lambda_L$, and since $S$ is a stochastic process and not a nonrandom function, its minimizer is not guaranteed to be a constant. Because of this, $\lambda_L$ is a random variable and $\lambdahat\distconv\lambda_L$ does not imply convergence in probability. The above therefore does not show consistency of $\lambdahat$ towards $\lambda_0$ or any fixed value for that matter, but rather that the distribution of $\lambdahat$ tends towards that of the minimizer of $S$. 

To illustrate the above, we performed a small simulation experiment for a hybrid LDA-logistic regression model with one covariate and binary response variable. We simulated from an LDA model, under which the true parameter minimizes both the limit of the generative and logistic likelihood. Hence $\theta_0(\lambda)$ is constantly equal to the true parameter in this simulation.  We used a $\psi(y,x,\theta) = \p{y-\exp\crl{\beta(\theta)^Tx}/\sqb{1+\exp\crl{\beta(\theta)^Tx}} }^2$, a smoothed version of the 0 - 1 loss function. With $n = 2000$ we computed $B = 1000$ different realizations of $S_n$. A selection of sample paths is displayed in \cref{fig:Sn simmulation} together with a histogram of the corresponding minima. From the figure we see that the sample paths are indeed not constant and that although the minimum is most often one, other values are repeatedly chosen, implying that $\lambda_L$ might really not be constant in this case.

\begin{figure}
	\centering
	\includegraphics[width=0.95\linewidth]{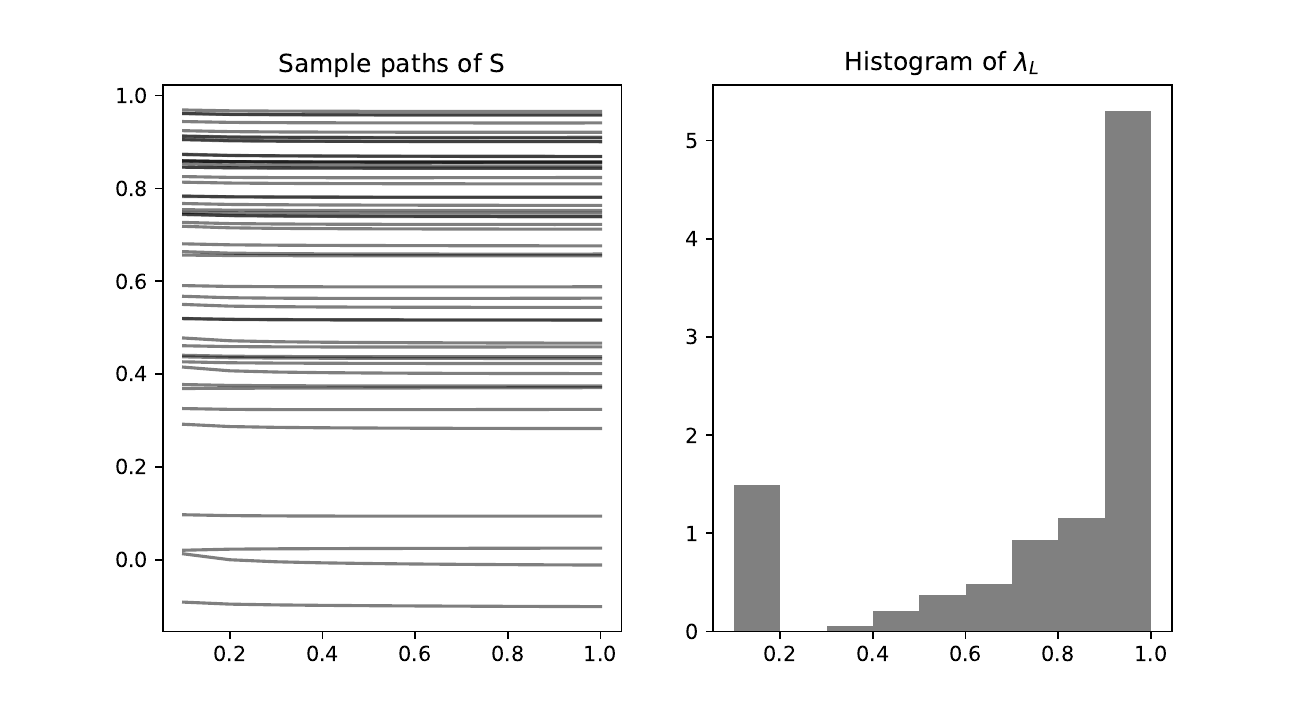}
	\caption{\label{fig:Sn simmulation}The figure shows two plots. To the left $B = 1000$ realizations of $S_n$ with $n = 2000$ are displayed. On the right hand side a histogram of the corresponding minima is shown.}
\end{figure}

Although the above arguments and simulations seem to indicate that there is reasonable doubt that $\lambdahat$ is consistent for any one value, this does not necessarily rule out all hope that $\sqrt{n}\crl{\thetahat(\lambdahat)-\theta_0}$ has a normal limiting distribution. In the simulation experiment described above, we also computed $\sqrt{n}\crl{\thetahat(\lambdahat)-\theta_0}$ in each iteration. Histograms of each component of this vector are displayed in \cref{fig:Thetahat simmulation}. From the plots we note that the distributions at least look quite normal. This intuition is supported by the QQ-plots shown in \cref{fig:QQplot simmulation}. Looking at these plots, we see no obvious deviances from the corresponding normal distributions. 

\begin{figure}
	\centering
	\includegraphics[width=0.95\linewidth]{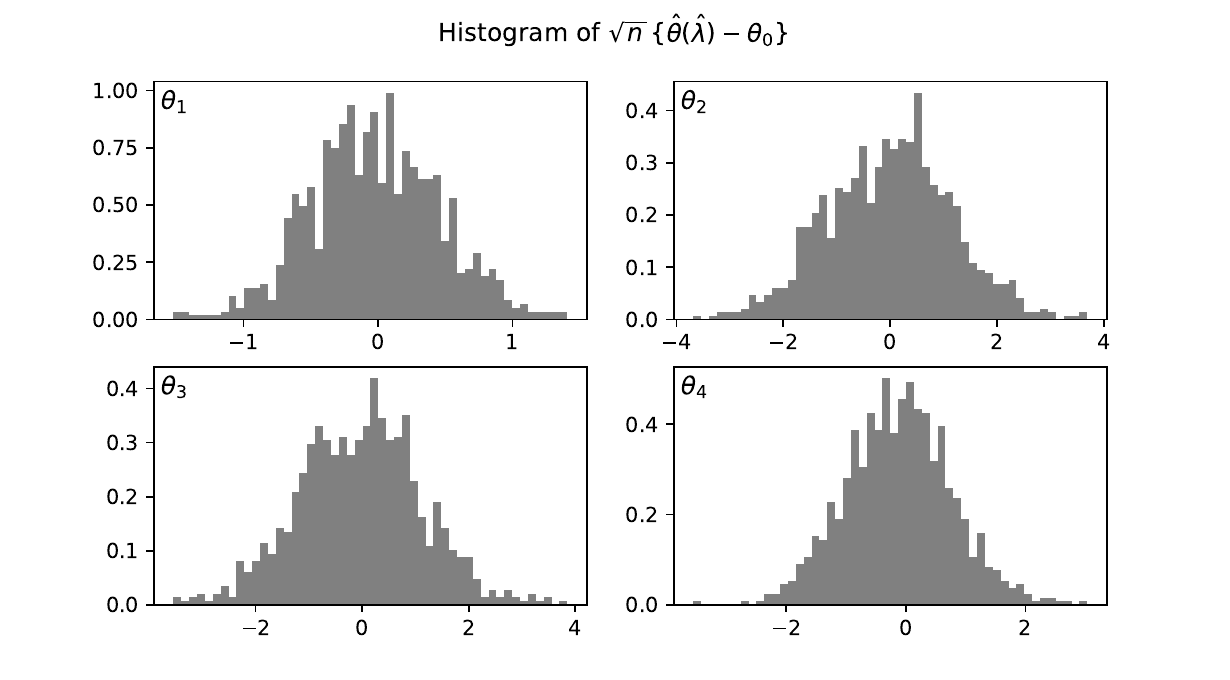}
	\caption{\label{fig:Thetahat simmulation}The figure shows histograms of each component of $\sqrt{n}\crl{\thetahat(\lambdahat)-\theta_0}$.}
\end{figure}

\begin{figure}
	\centering
	\includegraphics[width=0.95\linewidth]{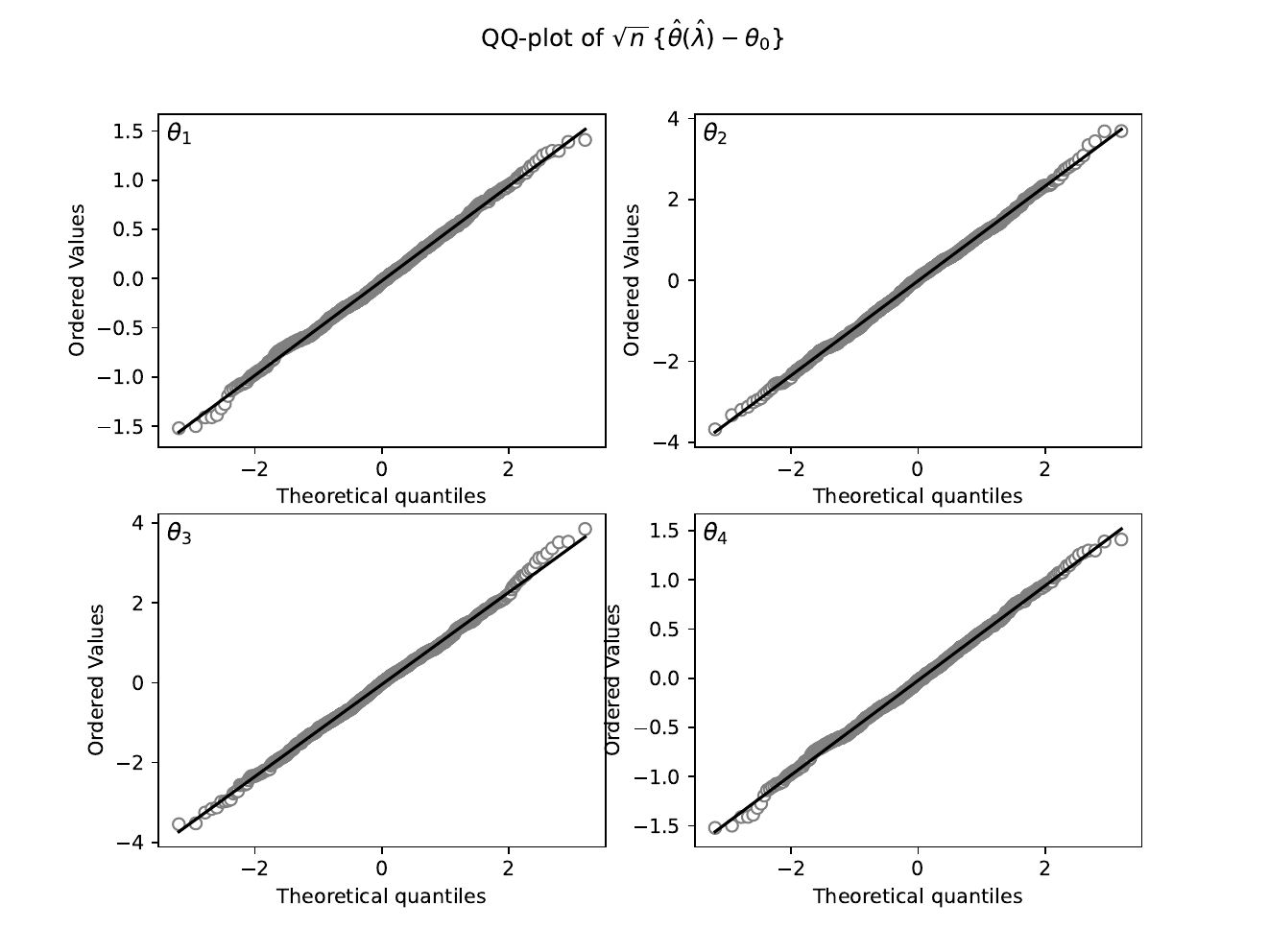}
	\caption{\label{fig:QQplot simmulation}The figure shows QQ-plots for each component of $\sqrt{n}\crl{\thetahat(\lambdahat)-\theta_0}$. The empirical quantiles are compared to those of fitted normal distributions.}
\end{figure}

One possible solution is to work with $A_n(\lambdahat) = \sqrt{n}K(\theta_0,\lambda_0)^{-1/2}J(\theta_0,\lambdahat)\crl{\thetahat(\lambdahat)-\theta_0}$ rather than $\sqrt{n}\crl{\thetahat(\lambdahat)-\theta_0}$. Pointwise, we have $A_n(\lambda)\distconv\text{N}\p{0,I_p}$, and hence, we would expect $A_n(\lambdahat)\distconv\text{N}\p{0,I_p}$ as well. To check whether this idea seems reasonable, we computed $A_n(\lambdahat)$ in each iteration of the simulation setting described above. A histogram of each component of $A_n(\lambdahat)$ is shown in \cref{fig:Thetahat normed simmulation} together with the graph of density function for the standard normal distribution. The plot seem to indicate that $A_n(\lambdahat)$ behaves more or less like a standard normally distributed variable.

\begin{figure}
	\centering
	\includegraphics[width=0.95\linewidth]{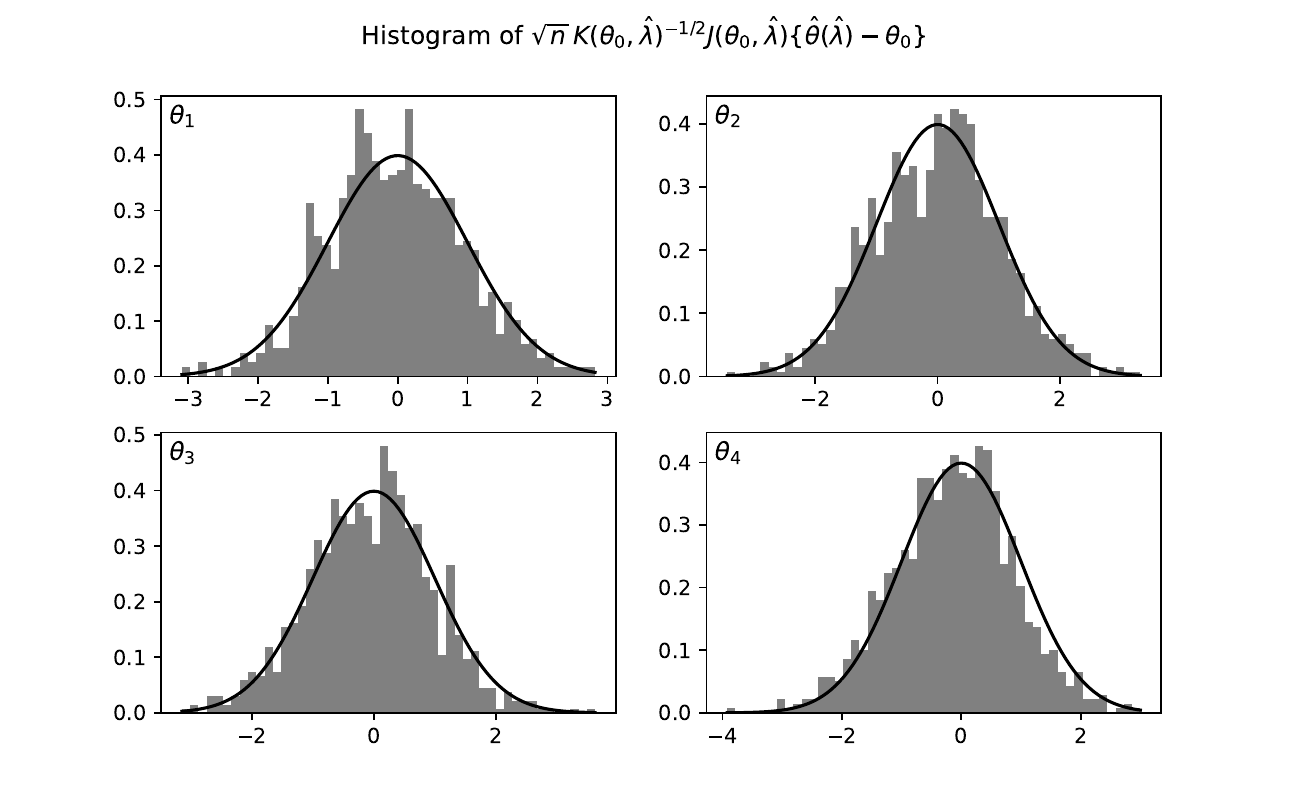}
	\caption{\label{fig:Thetahat normed simmulation}The figure shows histograms of each component of $\sqrt{n}K(\theta_0,\lambdahat)^{-1/2}J(\theta_0,\lambdahat)\crl{\thetahat(\lambdahat)-\theta_0}$.}
\end{figure}

One could argue that  Corollary 1 and Theorem 5 failing to cover situations where the limiting risk function is flat, makes the result and the estimators of Theorem 4 less useful in practical situations. Say that we for instance, have fitted a generative hybrid-discriminative model to some data and want to apply the theory in the main part of this paper to construct confidence intervals and hypothesis tests for the parameters in the model. If the generative model fitted to the data is wrongly specified, this can be done by applying Corollary 1 or Theorem 5 in conjunction with Theorem 4. If the model conditions do hold true, however, the arguments made in this section shows that the results cannot be applied. Hence, we cannot say that Corollary 1, Theorem 5 or the estimators in Theorem 4 are truly model-agnostic, since they fail to apply when the limiting risk function is flat. In theory, this is of course quite problematic. In practice, however, this is less of a problem. When $\theta_0'(\lambda) = 0$ for all values of $\lambda$, $\E\varphi(Z,\theta_0,\lambda)=0$ for all $\lambda$. Because of this, $\thetahat'(\lambdahat)$ of Theorem 3 is consistent for $0$ and $\hat{A}^*$ in Theorem 4 is asymptotically equivalent with $(J^{-1}(\lambdahat),0,0)$. Hence, $\hat{A}^*\hat{K}(\lambdahat)^*\p{\hat{A}^*}^T$ is asymptotically equivalent with $J^{-1}(\lambdahat)K(\lambdahat)J^{-1}(\lambdahat)$ which is what the simulations of this section seem to indicate is the limiting variance of $\sqrt{n}\crl{\thetahat(\lambdahat)-\theta_0}$.

Lastly, we would like to point out that there are situations where $\thetahat(\lambdahat)$ is consistent even if $\theta_0(\lambda)$ is a flat function, see e.g. \cite{arcones2005convergence} where the limit of the risk function used to select $\lambda$ is non-flat in $\lambda$ even when $\theta_0(\lambda)$ is constant. Cases where the limiting risk function is non-constant in $\lambda$ even though $\theta_0(\lambda)$ is flat are not discussed in this paper, but we point out that such estimators do exist and that for such cases the behavior discussed and illustrated in this section is not guaranteed.

\section{The effect of validation sets}\label{section:Validation set}
In this section we will give an informal discussion on how an alternative tuning scheme based on setting aside validation sets affects the asymptotic distribution of $\thetahat(\lambdahat)$. This problem is essentially equivalent to situations with auxiliary information, see \cite{qin2000miscellanea} which have inspired our proofs.

Assume we have i.i.d.~data $Z_1,\ldots,Z_n$ from some distribution $F$. In this section, we will assume that the first $n_1$ data points are used for tuning $\lambda$ and that the remaining $n_2 = n-n_1$ data points are used for estimating $\thetahat$ with the chosen value of $\lambda$, i.e.
\begin{equation*}
	\thetahat(\lambda) = \textup{ zero of}\sum_{i=n_1+1}^{n}\varphi(Z_i,\theta,\lambda)\textup{ and }\lambdahat = \underset{\lambda}{\textup{argmin}}\sum_{i=1}^{n_1}\psi[Z_i,\thetahat(\lambda)].
\end{equation*}
Arguing as in the paragraph preceding Theorem 1 in the main part of the paper, we get
\begin{equation*}
	\hat\alpha= \begin{pmatrix}
		\thetahat(\lambdahat)\\
		\lambdahat\\
		\thetahat'(\lambdahat)
	\end{pmatrix} = \textup{zero of }\begin{pmatrix}
	\sum_{i=n_1+1}^{n}\eta_1(Z_i,\theta,\lambda,D)\\
	\sum_{i=1}^{n_1}\eta_2(Z_i,\theta,\lambda,D)\\
	\sum_{i=n_1+1}^{n}\eta_3(Z_i,\theta,\lambda,D)
	\end{pmatrix},
\end{equation*}
where $\eta_1,\,\eta_2$ and $\eta_3$ are as defined in Theorem 1 in the main part of the paper.

By Lemma~\ref{lemma:Uniform consistency}, $\hat\alpha$ converges in probability towards $\alpha_0$, the zero of
\begin{equation*}
	\alpha\mapsto\begin{pmatrix}
		p\E\eta_1(Z,\alpha)\\
		(1-p)\E\eta_2(Z,\alpha)\\
		p\E\eta_3(Z,\alpha)
	\end{pmatrix},
\end{equation*}
provided $n_1/n\to p$ as $n\to\infty$, and the regularity conditions of Lemma~\ref{lemma:Uniform consistency} hold true. Since $p$ is constant, $\alpha_0$ is equal to the zero of $\alpha\mapsto\E[(\eta_1(Z,\alpha)^T, \eta_2(Z,\alpha)^T,\eta_3(Z,\alpha)^T)]^T$. Hence, $\hat\alpha$, and in particular $\thetahat(\lambdahat)$, is consistent also when tuning with respect to the error on a validation set.

For asymptotic normality, notice that Taylor's theorem ensures
\begin{equation*}
	0 = \begin{pmatrix}
		\sum_{i=n_1+1}^{n}\eta_1(Z_i,\hat\alpha)\\
		\sum_{i=1}^{n_1}\eta_2(Z_i,\hat\alpha)\\
		\sum_{i=n_1+1}^{n}\eta_3(Z_i,\hat\alpha)
	\end{pmatrix} =  \begin{pmatrix}
	\sum_{i=n_1+1}^{n}\eta_1(Z_i,\alpha_0)\\
	\sum_{i=1}^{n_1}\eta_2(Z_i,\alpha_0)\\
	\sum_{i=n_1+1}^{n}\eta_3(Z_i,\alpha_0)
	\end{pmatrix} + (\hat\alpha-\alpha_0)^T\begin{pmatrix}
	\sum_{i=n_1+1}^{n}J\eta_1(Z_i,\alpha^*)\\
	\sum_{i=1}^{n_1}J\eta_2(Z_i,\alpha^*)\\
	\sum_{i=n_1+1}^{n}J\eta_3(Z_i,\alpha^*)
	\end{pmatrix},
\end{equation*}
for some $\alpha^*$ on the line segment between $\hat\alpha$ and $\alpha_0$. Hence,
\begin{equation*}
	\sqrt{n}(\hat\alpha-\alpha_0) = -\left[\begin{pmatrix}
		\frac{1}{n}\sum_{i=n_1+1}^{n}J\eta_1(Z_i,\alpha^*)\\
		\frac{1}{n}\sum_{i=1}^{n_1}J\eta_2(Z_i,\alpha^*)\\
		\frac{1}{n}\sum_{i=n_1+1}^{n}J\eta_3(Z_i,\alpha^*)
	\end{pmatrix}^{-1}\right]^T\frac{1}{\sqrt{n}}\begin{pmatrix}
	\sum_{i=n_1+1}^{n}\eta_1(Z_i,\alpha_0)\\
	\sum_{i=1}^{n_1}\eta_2(Z_i,\alpha_0)\\
	\sum_{i=n_1+1}^{n}\eta_3(Z_i,\alpha_0)
	\end{pmatrix}.
\end{equation*}
Assuming $J\eta_j$ for $j=1,2,3$ are sufficiently regular, we therefore have
\begin{equation*}
	\sqrt{n}(\hat\alpha-\alpha_0) = -\left[\begin{pmatrix}
		(1-p)\E\,J\eta_1(Z,\alpha_0)\\
		p\E\,J\eta_2(Z,\alpha_0)\\
		(1-p)\E\,J\eta_3(Z,\alpha_0)
	\end{pmatrix}^{-1}\right]^T\begin{pmatrix}
		(1-p)^{1/2}n_2^{-1/2}\sum_{i=n_1+1}^{n}\eta_1(Z_i,\alpha_0)\\
		p^{1/2}n_1^{-1/2}\sum_{i=1}^{n_1}\eta_2(Z_i,\alpha_0)\\
		(1-p)^{1/2}n_2^{-1/2}\sum_{i=n_1+1}^{n}\eta_3(Z_i,\alpha_0)
	\end{pmatrix}+\opr(1).
\end{equation*}

The above shows that $\sqrt{n}(\alphahat-\alpha_0)$ is a linear combination of terms which converge towards normally distributed variables at the speed of $O(n_1^{-1/2}) = \Opr(pn^{-1/2})$ or $O(n_2^{-1/2}) = \Opr((1-p)n^{-1/2})$. Assuming $0<p<1$, the convergence rate is therefore reduced from a speed of $n^{-1/2}$ to $pn^{-1/2}$ or $(1-p)n^{-1/2}$ when using a validation set to tune $\lambda$ rather than cross validation. This is a relatively natural result as $np$ and $n(1-p)$ are approximately equal to the amount of data which is set aside for fitting the model and tuning $\lambda$ respectively.

\bibliography{references}